\documentclass[12pt,a4paper,reqno]{amsart}
\usepackage{amsmath}
\usepackage{amsfonts}
\usepackage{amssymb}
\usepackage[dvipsnames]{xcolor}
\usepackage{mathrsfs}
\usepackage{hyperref}
\usepackage{cleveref}
\usepackage{enumitem}

\usepackage{cite}
\usepackage[foot]{amsaddr}
\DeclareMathOperator*{\esssup}{ess\,sup}

\newtheorem{theorem}{Theorem}[section]
\newtheorem{corollary}[theorem]{Corollary}
\newtheorem{lemma}[theorem]{Lemma}
\newtheorem{proposition}[theorem]{Proposition}
\newtheorem{definition}[theorem]{Definition}
\newtheorem{remark}[theorem]{Remark}

\numberwithin{equation}{section}

\begin{document}
	
	\begin{center}
		\Large{\textbf{Well-posedness and large time behavior of a size-structured growth-coagulation-fragmentation model}}
	\end{center}
	\medskip
	%\centerline{by}
	\medskip
	\centerline{${\text{Saroj ~ Si}}$ and ${\text{Ankik Kumar Giri$^{*}$}}$ }\let\thefootnote\relax\footnotetext{$^{*}$Corresponding author. Tel +91-1332-284818 (O);  Fax: +91-1332-273560  \newline{\it{${}$ \hspace{.3cm} Email address: }}ankik.giri@ma.iitr.ac.in}
	\medskip
	{\footnotesize
		\centerline{Department of Mathematics, Indian Institute of Technology Roorkee,}
		
		\centerline{Roorkee-247667, Uttarakhand, India}
	}

	\bigskip

\begin{quote}
		{\small {\em \bf Abstract.} The existence and uniqueness of weak solutions to a size-structured growth-coagulation-fragmentation (GCF) equation with a renewal boundary condition are shown for a class of unbounded coagulation and fragmentation kernels. The existence proof is based on a weak compactness framework in the weighted $L^1$-space. This result extends the existence results of Banasiak and Lamb \cite{banasiak2009coagulation} and Ackleh et al. \cite{ackleh2003first, ackleh2021structured}. Furthermore, we establish a stability result and derive uniqueness as a direct consequence of it. Moreover, this study explores the large time behavior of weak solutions.}
\end{quote}
	
	\vspace{.3cm}
	
	\section{introduction}
	%%%%%%%%%%%%%%%%
	%%%%%%%%%%%%%%%%
	\noindent Coagulation and fragmentation models have numerous applications in physical, chemical, biological, and environmental sciences, see \cite{Lissauer1993, Safronov1972, Ziff1980, Wells2018, Friedlander2000, Adler1997, BLL_book} and the references therein. In particular, the size-structured growth-coagulation-fragmentation (GCF) model~\cite{ackleh1997parameter, ackleh1997modeling, Adler1997, rudnicki2006fragmentation} with a renewal boundary condition describes the evolution of phytoplankton populations in the marine environment. Phytoplankton~\cite{nowicki2022quantifying, smetacek1985role, alldredge1989direct, alldredge1988situ} play a crucial role in climate change by controlling atmospheric carbon dioxide levels through the process of carbon deposition in the deep ocean. Additionally, they are vital to oceanic and freshwater food chains~\cite{hjort1914fluctuations, cushing1990plankton, cushing1983fish}, serving as the sole food source for many fish species during their larval stage. Therefore, the study of the size-structured GCF model with a renewal boundary condition is of fundamental importance. Denoting by $\xi(t,u)$ the density of aggregates of size $u$ at time $t$, the size-structured GCF equations~\cite{ackleh1997parameter} are given by 
	\begin{align}	
		\partial_{t}\xi(t,u) &=-\partial_{u}(g\xi)(t,u)-\mu(u)\xi(t,u) -\alpha(u)\xi(t,u)+\int_{u}^{\infty} \alpha(u_{1})\beta(u|u_{1})\xi(t,u_{1})\;du_{1}\nonumber\\
		&\quad +\frac{1}{2} \int_{0}^{u} \Upsilon(u-u_{1},u_{1})\xi(t,u-u_{1})\xi(t,u_{1})\;du_{1} \nonumber\\
		&\quad-\int_{0}^{\infty} \Upsilon(u,u_{1})\xi(t,u)\xi(t,u_{1})\;du_{1},\quad (t,u)\in (0,\infty)^{2},\label{maineq.1}
	\end{align}
supplemented with the initial condition
	\begin{align}\label{initialdata}
	  \xi(0,u)=\xi^{{\mathrm{in}}}(u) \geq 0, \quad u\in [0,\infty),
	\end{align}
	and the renewal boundary condition
	\begin{align}\label{maineq.2} 
	\lim_{u\rightarrow 0^{+}}g(u)\xi(t,u)&=\int_{0}^{\infty}a (u_{1})\xi(t,u_{1})\;du_{1},\quad t>0.
   \end{align}
In~\eqref{maineq.1}--\eqref{maineq.2}, $g$ represents the rate of formation of aggregates of size $u$ due to the growth of cells by division within it, while the parameter $a$ denotes the rate at which daughter cells are incorporated into the single-cell population. The death rate $\mu$ represents the death or removal of cells. The daughter distribution function $\beta(u|u_{1})$ describes the distribution of daughter clusters of size $u$ produced by a mother cell of size $u_{1}$. Moreover, the interaction rates $\alpha$ and $\Upsilon$ are non-negative functions, referred to as the fragmentation and coagulation kernels, respectively. In equation \eqref{maineq.1}, the third term on the right-hand side represents the loss of clusters of size $u$ due to their breakup into smaller fragments. The sixth term accounts for the loss of clusters of size $u$ as they merge with other clusters through coalescence. Conversely, the fourth term describes the formation of clusters of size $u$ resulting from the breakup of larger clusters, while the fifth term represents the formation of clusters of size $u$ due to the merging of smaller clusters. For future use, we define the following fragmentation operator $\mathcal{F}$ and the coagulation operator $\mathcal{C}$ as 
\begin{align}\nonumber
	\mathcal{F}(\xi)(t,u):=-\alpha(u)\xi(t,u)+\int_{u}^{\infty} \alpha(u_{1})\beta(u|u_{1})\xi(t,u_{1})\;du_{1}
\end{align}
and 
\begin{align}
	\mathcal{C}(\xi)(t,u)&:=\frac{1}{2} \int_{0}^{u} \Upsilon(u-u_{1},u_{1})\xi(t,u-u_{1})\xi(t,u_{1})\;du_{1} \nonumber \\
	&\quad	-\int_{0}^{\infty} \Upsilon(u,u_{1})\xi(t,u)\xi(t,u_{1})\;du_{1},\nonumber 
\end{align}
for $(t,u)\in (0,\infty)^{2}$.

Jackson \cite{jackson1990model} initially introduced an algal growth population model incorporating the  coagulation term in a discrete setting. His approach required either that the cells within aggregates be dead or that all offspring cells remain within the aggregate after cell division. In \cite{ackleh1997modeling}, Ackleh and Fitzpatrick proposed the coagulation term for a continuous size-structured population model. Their model included biological processes such as coalescence, birth, death, and growth. Later, in \cite{ackleh1997parameter}, Ackleh extended this model by adding a fragmentation term. Compared to the GCF equations~\eqref{maineq.1}--\eqref{initialdata}~\cite{Banasiak2012, BLL_book, GajewskiZacharias1982, Gajewski1983, FriedmanReitich1990, banasiak2020growth, banasiak2022explicit, Giri2025well}, the GCF equations incorporating the renewal boundary condition \eqref{maineq.2}, have received relatively less attention in the mathematical literature for unbounded coagulation kernels. The first mathematical study of~\eqref{maineq.1}--\eqref{maineq.2} was carried out by Ackleh and Fitzpatrick~\cite{ackleh1997modeling}, where the existence of a unique solution was shown in the Banach space $L^{2}(y_{0}, y_{1})$ for the non-negative parameters $g$, $\mu$, $\Upsilon$, and $a$ satisfying the following conditions
\begin{align*}
   \mu, a \in L^{\infty}(y_{0}, y_{1}),\; g\in W^{1,\infty}(y_{0}, y_{1})\; \text{with}\; g(u)>0 \; \text{on}\;(y_{0}, y_{1})\; \text{and}\;\;g(y_{1})=0,\\
	 \Upsilon \in L^{\infty}([y_{0}, y_{1}]^{2}), \; 0\leq \Upsilon(u,u_{1})=\Upsilon(u_{1},u), \; \text{and}\; \Upsilon(u,u_{1})=0 \; \text{ for }\; u+u_{1}> y_{1},
\end{align*}
for $(u,u_{1})\in [y_{0}, y_{1}]^{2}$, $y_{0}, y_{1}\in (0,\infty)$ with $y_{0}< y_{1}$. Later, Ackleh~\cite{ackleh1997parameter} demonstrated the existence of a unique solution to \eqref{maineq.1}--\eqref{maineq.2} by considering the binary fragmentation in the Banach space $L^{1}(y_{0}, y_{1})$ for the non-negative parameters satisfying the following conditions 
  \begin{align*}
  	g\in C^{1}([y_{0}, y_{1}]),\; \mu, a \in C([y_{0}, y_{1}]), \Upsilon \in C([y_{0}, y_{1}]^{2}),\\ \Upsilon(u,u_{1})=\Upsilon(u_{1},u),\;  \text{and}\; \Upsilon(u,u_{1})=0 \; \text{ for }\; u+u_{1}> y_{1}.
  \end{align*}
  for $(u,u_{1})\in [y_{0}, y_{1}]^{2}$, $y_{0}, y_{1}\in (0,\infty)$ with $y_{0}< y_{1}$. Afterward, Ackleh and Deng~\cite{ackleh2003first} extended the existence and uniqueness results to the unbounded domain $(0,\infty)$ using the monotone sequence method for the class of bounded coagulation kernels and identically zero fragmentation kernel, i.e., $\alpha\equiv 0$. Further, Banasiak and Lamb~\cite{banasiak2009coagulation} extended previous existence and uniqueness results by considering an unbounded fragmentation kernel $\alpha$, a breakage rate $\beta$, a birth rate $a$, and a growth rate $g$. In particular, $\alpha$, $\beta$, $a$,  and $g$ satisfy \eqref{alphaassum}, \eqref{particlebound}, \eqref{lmassconservation}, \eqref{birthrateassum}, and \eqref{growthassum} with
	\begin{align}\label{gintegrability}
		\int_{0^{+}}\frac{du}{g(u)}<\infty, 
	\end{align}
where $\int_{0^{+}}$represents the integral over a right-neighborhood of $u=0$. In this study, Banasiak and Lamb~\cite{banasiak2009coagulation} considered only a class of bounded coagulation kernels $\Upsilon$ and a bounded death rate $\mu$. In \cite{banasiak2011blow}, Banasiak has shown the blowup of the first moment $M_{1}(\xi)$ of the solution~(as defined in \eqref{momentdefn} for $\gamma=1$) to~\eqref{maineq.1}--\eqref{maineq.2} when the growth of the birth rate $a$ is sufficiently large (at least quadratic) as $u\to \infty$. Later, Ackleh et al.~\cite{ackleh2021structured, ackleh2023high, ackleh2023finite} established the existence of a unique measure-valued solution in the space of Radon measures for a class of bounded and globally Lipschitz coagulation kernels $\Upsilon$ and for a fragmentation kernel $\alpha$ belonging to $W^{1,\infty}(0,\infty)$. In~\cite{banasiak2024growth}, Banasiak et al. studied the well-posedness and large time behavior of solutions to \eqref{maineq.1}--\eqref{maineq.2} for a class of unbounded fragmentation kernels, assuming the coagulation kernel  $\Upsilon\equiv 0$.

Based on the available literature, all prior studies~\cite{ackleh1997modeling, ackleh1997parameter, ackleh2003first, banasiak2009coagulation, banasiak2011blow, ackleh2021structured, ackleh2023high, ackleh2023finite, banasiak2024growth} on \eqref{maineq.1}--\eqref{maineq.2} have relied on the hypothesis that the coagulation kernel $\Upsilon$ is bounded and that the growth rate $g$ satisfies \eqref{gintegrability}. This work aims to relax these assumptions, allowing for the study of a more general class of physically valid unbounded coagulation kernels and growth rate $g$.

The main goal of this paper is to establish the existence and uniqueness of a weak solution to \eqref{maineq.1}--\eqref{maineq.2} in the weighted $L^{1}$-space for a class of unbounded coagulation kernels $\Upsilon$, fragmentation kernels $\alpha$, birth rates $a$, death rates $\mu$, and growth rates $g$. In particular, we consider unbounded coagulation kernels, including the shear coagulation and gravitational coagulation kernels, as presented in \cite[Table 1]{aldous1999deterministic} and \cite[Table 1]{smit1994aggregation}. An important aspect that has received little attention is the long-term behavior of solutions to the GCF model \eqref{maineq.1}--\eqref{maineq.2}. While the asymptotic properties of linear growth-fragmentation equations have been well studied in \cite{canizo2021spectral, banasiak2024growth}, similar results for the GCF model remain limited. In this work, we address this gap by analyzing the long-term behavior of weak solutions to \eqref{maineq.1}--\eqref{maineq.2}. Our study is mainly inspired by prior works \cite{laurenccot2007well, Laurencot2001, ackleh2003first, banasiak2009coagulation, Giri2025well, lachowicz2003oort}, with the goal of broadening the class of admissible coagulation kernels.

We now provide an overview of this paper: The next section introduces the assumptions on the parameters $\Upsilon$, $\beta$, $\mu$, $\alpha$, $g$, $a$, and $\xi^{{\mathrm{in}}}$. It also presents the notion of a weak solution to \eqref{maineq.1}--\eqref{maineq.2} and our main results. The existence of global weak solutions to \eqref{maineq.1}--\eqref{maineq.2} is addressed in \Cref{existence section}. First, we define the approximate system~\eqref{truncated-1}--\eqref{truncated-2} corresponding to~\eqref{maineq.1}--\eqref{maineq.2}. Then, we establish the global existence of a unique classical solution to~\eqref{truncated-1}--\eqref{truncated-2} and derive bounds on the zeroth moment~$M_{0}(\xi)$ and the first moment~$M_{1}(\xi)$ of the solution, as defined in~\eqref{momentdefn}. Additionally, this section includes the method of weak $L^{1}$-compactness technique originally introduced by Stewart in \cite{Stewart1989} and later adapted by Lauren{\c{c}}ot in \cite{Laurencot2001}. Specifically, we first control the tail of the approximate sequence of solutions. Then by using the uniform integrability result along with the Dunford-Pettis theorem, we show that the approximate sequence of solutions is weakly compact in $L^1([0,\infty))$. Subsequently, we show the time equicontinuity to apply a generalized form of the Arzelà-Ascoli theorem \cite[Theorem 1.3.2]{Vrabie2003} to extract a convergent sub-sequence. Next, we improve the convergence in~\Cref{improved convergence}~to show that the limit of the convergent sub-sequence is the required weak solution to \eqref{maineq.1}--\eqref{maineq.2}. In \Cref{uniqueness section}, we turn our attention to establishing a stability result that implies the uniqueness of the weak solution in some appropriate space. Finally, in the last section, large time behavior of weak solutions is discussed in the spirit of \cite[Theorem 1.1]{escobedo2002gelation} and \cite[Theorem 2.5 and Proposition 4.1]{lachowicz2003oort}. 
%%%%%%%%%%%%%%%%%%%%%%%%%%%%%%%%%%%%%%%%%%%%%%%%%%%%%%%%%%%%%%%%%%%%%%%%%%%%%%%%%%%%%%%%%%%%%%%%%%%%%%%%%%
\section{Hypotheses and Main Results}\label{Assumption and Main Result}
		\noindent We make the following hypotheses on the coefficients $\Upsilon$, $\beta$, $\alpha$, $\mu$, $g$, $a$, and the initial data $\xi^{{\mathrm{in}}}$ throughout this paper.
		\begin{enumerate}[label=(\roman*)]
			\item $\Upsilon$ is a real-valued, measurable function on $(0,\infty)^2$. There is $\Upsilon_{0}>0$ such that
			\begin{align}\label{multiplicativecoagassum}
				0\leq \Upsilon(u_{1},u)=\Upsilon(u,u_{1}) \leq \Upsilon_{0}(1+u)(1+u_{1})\quad \text{ for all}\quad (u,u_{1})\in (0,\infty)^{2}.
			\end{align}
		We consider two distinct cases. In the first case, we assume that there exists a constant $\Upsilon_{1}>0$ such that
			\begin{align}\label{additivecoagassum}
				\Upsilon(u,u_{1}) \leq \Upsilon_{1}(u+u_{1}), \quad (u,u_{1})\in (1, \infty)^{2}.
			\end{align}
	In the second case, we assume instead that 
			\begin{align}\label{subquadraticcoagassum}
				\Upsilon(u,u_{1}) \leq r(u)r(u_{1}), \quad (u,u_{1})\in (0,\infty)^{2},
		\end{align}
	where $r : (0,\infty) \to (0,\infty)$ is a measurable function satisfying
		\begin{align}\label{sublimboundcoagassum}
			\sup_{u>0}\frac{r(u)}{1+u}<\infty\quad \text{and}\quad \lim_{u\rightarrow \infty}\frac{r(u)}{u}=0.
		\end{align}
	\begin{remark}
		It is important to note that the assumptions on the coagulation kernels include several physically significant cases. Specifically, the following coagulation kernels are covered, as documented in \cite[Table 1]{aldous1999deterministic}, \cite[Table 1]{smit1994aggregation}, \cite[Appendix B.5]{giri2010mathematical}, and the references therein.
		
		\begin{table}[h]
			\centering
			\renewcommand{\arraystretch}{1.5}
			\begin{tabular}{|c|l|l|}
				\hline
				\textbf{No.} & \textbf{Kernel Type} & \textbf{Expression} \\
				\hline
				1 & Linear shear kernel & $(u^{1/3}+u_{1}^{1/3})^{3}$ \\
				2 & Nonlinear shear kernel & $(u^{1/3}+u_{1}^{1/3})^{7/3}$ \\
				3 & Gravitational kernel & $(u^{1/3}+u_{1}^{1/3})^{2}\left|u^{1/3}-u_{1}^{1/3}\right|$ \\
				4 & Modified Smoluchowski kernel & $\frac{(u^{1/3}+u_{1}^{1/3})^{2}}{u^{1/3}u_{1}^{1/3}+c}$ \\
				5 & Activated sludge flocculation kernel & $\frac{(u^{1/3}+u_{1}^{1/3})^{q}}{1+\left(\frac{u^{1/3}+u_{1}^{1/3}}{2u_{c}^{1/3}}\right)^{3}}$ \\
				6 & Product kernel & $(uu_{1})^{\omega}$ \\
				\hline
			\end{tabular}
			\caption{Standard coagulation kernels}
			\label{tab:coagulation_kernels}
		\end{table}
		In \Cref{tab:coagulation_kernels}, $c$ and $u_{c}$ are positive constants with $0\leq q<3$ and $0\leq \omega <1$. The coagulation kernels $(1)–(5)$ satisfy conditions \eqref{multiplicativecoagassum} and \eqref{additivecoagassum}, whereas the coagulation kernel $(6)$ satisfies \eqref{multiplicativecoagassum}, \eqref{subquadraticcoagassum}, and \eqref{sublimboundcoagassum}.
	\end{remark}
     	\item The daughter distribution function $\beta$ is measurable, non-negative, and satisfies $\beta(u|u_{1}) = 0$ whenever $u > u_{1}$. Additionally, we impose a constraint on the maximum number of daughter particles
     	\begin{align}\label{particlebound}
     		\sup_{0 \leq u_{1} < \infty} n(u_{1}) = \sup_{0 \leq u_{1} < \infty} \int_{0}^{u_{1}} \beta(u|u_{1}) \, du = M < \infty,
     	\end{align}
     	and 
     	\begin{align}\label{lmassconservation}
     		\int_{0}^{u_{1}} u \beta(u|u_{1})\;du = u_{1},\quad u_{1}>0,
     	\end{align}
     	which reflects the property of local mass conservation.
Moreover, we suppose that the fragmentation kernel $\alpha$ is measurable on $[0,\infty)$ and satisfies
\begin{align}  
	 0&\leq \alpha(u)\leq P_{\alpha}u+Q_{\alpha},\quad u\in [0,\infty),\label{alphaassum}
\end{align}
where $P_{\alpha}$ and $Q_{\alpha}$ are positive real numbers. 
Under the condition~\eqref{alphaassum}, the fragmentation kernel $\alpha$ satisfies the following property
\begin{align}\label{localbound}
	\alpha\in L^{\infty}([0, n]) \quad \text{ for all}\quad n>0.
\end{align}
The assumptions imposed on $\beta$ and $\alpha$ encompass the following physically relevant classes, as discussed in \cite{vigil1989stability} and \cite[Section 2.3]{BLL_book}:  
\begin{align*}  
	\alpha(u) = l_{0}u^{l_{1}}, \quad u \ge 0\quad \text{and} \quad \beta(u|u_{1}) = (\nu +2) \frac{u^{\nu}}{u_{1}^{\nu +1}}, \quad 0 < u < u_{1},  
\end{align*}  
where $l_{0} \geq 0$, $0 \leq l_{1} \leq 1$, and $-1 < \nu \leq 0$.

\item The function $\mu$ is real-valued and measurable on $[0,\infty)$ and satisfies
\begin{align}\label{mu}
	0&\leq \mu(u)\leq P_{\mu}u+Q_{\mu},\quad u\in [0,\infty),
\end{align}
where $P_{\mu}$ and $Q_{\mu}$ are positive real numbers.  			
\item $g$ is an absolutely continuous function on $[0,\infty)$, i.e., $g\in AC([0,\infty))$ with
	\begin{align}\label{growthassum}
	0\leq g(u)\leq g_{1}u+g_{0}, \quad u\in [0,\infty),	
	\end{align}
			where $g_{1}$ and $g_{0}$ are fixed positive real numbers.
	\item
	We assume that
	\begin{align}\label{birthrateassum}
		0\leq a(u) \leq a_{1}(1+u), \quad u\in [0,\infty),
	\end{align}
for some constant $a_{1}>0$. 
\item			
The initial data $\xi^{{\mathrm{in}}}$ is assumed to satisfy  
\begin{align}  
	\xi^{{\mathrm{in}}} \in Y_{+}, \label{initialassum}  
\end{align}  
where $Y_{+}$ is the non-negative part of the complete normed linear space $Y$, defined as  
\begin{align}  
	Y = L^{1}\left([0, \infty); (1+u)\;du\right)\nonumber  
\end{align}  
and equipped with the norm $\|\cdot\|$, expressed as
\begin{align}  
	\|\xi\| = \int_{0}^{\infty} (1+u)|\xi(u)|\;du, \quad \xi \in Y.\nonumber  
\end{align} 
In other words, $Y_{+}=\left\{\xi\in Y \;|\; \xi\geq 0 \text{ a.e.} \right\}$.
\end{enumerate}

To facilitate our analysis, we first define the moment of order $\gamma\in \mathbb{R}$ for $\xi(t)$ as
	\begin{align}\label{momentdefn}
		M_{\gamma}(\xi(t)) &:= \int_{0}^{\infty} u^{\gamma}\xi(t,u) \,du, \quad t\ge 0.
	\end{align}
In addition, we introduce the following functional spaces for future use
\begin{align*}
	Y_{\infty}:=\left\{ \eta : [0, \infty) \rightarrow [0, \infty)\;\mid\; \eta~\text{is measurable and } \|\eta\|_{\infty}<\infty\right\}
\end{align*}
and
\begin{align*}
	Y_{r}:=\left\{\eta : [0, \infty) \rightarrow [0, \infty)\;\mid\: \eta~\text{is measurable and } 0\leq \eta(u)\leq r(u), \; u\in (0, \infty)\right\},
\end{align*}
where
\begin{align*}
	\|\eta\|_{\infty}:=\esssup_{0\leq u<\infty}\frac{|\eta(u)|}{1+u},
\end{align*}
and the function $r$ is defined in \eqref{subquadraticcoagassum} and \eqref{sublimboundcoagassum}. Clearly, the assumptions \eqref{alphaassum}, \eqref{mu}, \eqref{growthassum}, and \eqref{birthrateassum} ensure that $\alpha$, $\mu$, $g$, and $a$ belong to $Y_{\infty}$. We now define the space of weakly continuous function space as 
\begin{align*}
	C\left([0,T]; w-X\right):=\left\{f\; | \;f: [0,T] \rightarrow X \;\text{is weakly continuous}\right\},
\end{align*}
where $T\in (0, \infty)$ and $X$ is a complete normed linear space.

Before providing our main outcomes of the paper, we recall the following definition of a weak solution to \eqref{maineq.1}--\eqref{maineq.2}, which will be used in the subsequent sections.
	\begin{definition} \label{definitionofsolution}
		Consider a class of coagulation kernels $\Upsilon$ satisfying~\eqref{multiplicativecoagassum}, and let $\beta$, $\mu$, $\alpha$, $g$, $a$, and $\xi^{{\mathrm{in}}}$  be the six parameters satisfying \eqref{particlebound}--\eqref{initialassum}. For $T>0$, a weak solution to \eqref{maineq.1}--\eqref{maineq.2} is a function $\xi : [0,T]\mapsto Y_{+}$ such that
		\begin{align}
			\xi \in C\left([0,T];w-L^1([0, \infty);(1+u)du)\right) \cap L^{\infty}(0,T; Y_{+}) \label{weak formulation-0} 
		\end{align}
		and satisfies the following for every $t \in [0, T]$ and $\vartheta \in C^{1}_c\left([0,\infty)\right)$
		\begin{align} \label{weak formulation}
			\int_{0}^{\infty} \xi(t,u)\vartheta(u)\;du  &=   	\int_{0}^{\infty} \xi^{{\mathrm{in}}}(u)\vartheta(u)\;du   
			+\int_{0}^{t}\int_{0}^{\infty} \vartheta^{\prime}(u)g(u)\xi(s,u)\;duds\nonumber \\ &+\vartheta(0)\int_{0}^{t}\int_{0}^{\infty} a(u)\xi(s,u)\;duds-\int_{0}^{t}\int_{0}^{\infty} \mu(u)\vartheta(u)\xi(s,u)\;duds \nonumber \\
			&+\int_{0}^{t}\int_{0}^{\infty}\psi(\vartheta)(u_{1}) \alpha(u_{1})\xi(s,u_{1})\;\;du_{1}ds  \nonumber \\
			&+\frac{1}{2}\int_{0}^{t}\int_{0}^{\infty} \int_{0}^{\infty}\tilde{\vartheta}(u,u_{1}) \Upsilon(u,u_{1})\xi(s,u) \xi(s,u_{1})\;dudu_{1}ds,
		\end{align}
	where 
	\begin{align*}
		\psi(\vartheta)(u_{1})=\int_{0}^{u_{1}}\vartheta(u)\beta(u|u_{1})\;du-\vartheta(u_{1})
	\end{align*}
	and
		\begin{align}
			\tilde{\vartheta}(u,u_{1}):=\vartheta(u+u_{1})-\vartheta(u)-\vartheta(u_{1}).  \nonumber
		\end{align}
 If $\vartheta \in C[0,\infty)\cap W^{1,\infty}(0,\infty)$, then $\left|\vartheta^{\prime} g\right|\leq g(u)\|\vartheta\|_{ W^{1,\infty}}\leq  \|\vartheta\|_{ W^{1,\infty}}\max\left\{g_{1},g_{0}\right\}(1+u)$. Owing to \eqref{weak formulation-0}, all the terms in \eqref{weak formulation} are finite.
	\end{definition}
We now present the main results of the paper.
	\begin{theorem} \label{main theorem}
		Suppose $\Upsilon$, $\beta$, $\alpha$, $\mu$, $g$, $a$, and $\xi^{{\mathrm{in}}}$ satisfy \eqref{multiplicativecoagassum}, \eqref{particlebound}, \eqref{lmassconservation}, \eqref{alphaassum}, \eqref{localbound}, \eqref{mu},  \eqref{growthassum}, \eqref{birthrateassum}, and \eqref{initialassum}.
		\begin{enumerate}[label=(\alph*)]
			\item If $\Upsilon$ also satisfies \eqref{additivecoagassum}, then at least one weak solution $\xi$ to \eqref{maineq.1}--\eqref{maineq.2}, as defined in \Cref{definitionofsolution}, exists. 
			\item If $\Upsilon$ also satisfies \eqref{subquadraticcoagassum} and \eqref{sublimboundcoagassum}, and if $g$, $\mu$, $a$, and $\alpha$ belong to $Y_{r}$, then at least one weak solution $\xi$ to \eqref{maineq.1}--\eqref{maineq.2}, as defined in \Cref{definitionofsolution}, exists.
	  \end{enumerate}
		\end{theorem}
To establish \Cref{main theorem}, we first utilize a classical existence result from \cite[Theorem 3.5]{banasiak2009coagulation} for the approximate system~\eqref{truncated-1}--\eqref{truncated-2} and subsequently apply a weak $L^{1}$-compactness method in the space $L^1([0,\infty);(1+u)du)$. The weak $L^{1}$-compactness approach, originally developed by Stewart in \cite{Stewart1989}, was later refined in \cite{Laurencot2001, Giri2025well} for the coagulation equation with a growth or transport term.

The subsequent task in establishing the well-posedness of \eqref{maineq.1}–\eqref{maineq.2} within an appropriate analytical framework is to examine the uniqueness of the weak solution. To this end, we derive the following result. 
	\begin{theorem} \label{uniqueness theorem}
		Let all the hypotheses of \Cref{main theorem}$(a)$ be satisfied. Suppose that $\xi^{{\mathrm{in}}} \in L^1\left([0,\infty); u^{2}du\right)$, $\mu \in L^{\infty}([0,\infty))$, and assume that $g$ satisfies the following Lipschitz condition with Lipschitz constant $A>0$ 
		\begin{align}
			|g(u)-g(u_{1})|& \leq A |u-u_{1}|\quad \text{ for }\;  (u,u_{1})\in  [0,\infty)^2.  \label{Lipschitz growth}
		\end{align} 
	Then, a unique weak solution $\xi$ to \eqref{maineq.1}--\eqref{maineq.2}, as defined in \Cref{definitionofsolution}, exists and satisfies  
	\begin{align}\label{solsecondmomentbound}  
		\xi \in L^{\infty}([0,T];L^1([0, \infty);(1+u^2)du))  
	\end{align}  
	for all $T>0$.
	\end{theorem}
	As anticipated, the proof of~\Cref{uniqueness theorem} follows a similar strategy to those used for uniqueness results in coagulation-fragmentation and growth-coagulation equations; see \cite{Stewart1990, giri2013uniqueness, EMRR2005, Giri2025well} and the references therein. To prove uniqueness, we first establish a stability result in \Cref{continuousdependence result}. The main challenge in this proof stems from the lack of differentiability of the weak solution $\xi$. Therefore, we employ a suitable regularization technique, following the theory developed by DiPerna and Lions~\cite{diperna1989ordinary}, \cite[Appendix~6.1 and~6.2]{perthame2006transport}, and later adapted by Giri et.al in~\cite[Theorem 2.4]{Giri2025well}.
	
	Next, we analyze the long-time behavior of weak solutions to~\eqref{maineq.1}--\eqref{maineq.2}. In the classical Smoluchowski coagulation model~\cite{escobedo2003gelation, escobedo2002gelation} and the Oort-Hulst-Safronov coagulation model~\cite{lachowicz2003oort}, it has been shown that the zeroth moment $M_{0}(\xi(t))$ and the first moment $M_{1}(\xi(t))$ of the solution $\xi$ tend to zero as $t \to \infty$ when the coagulation kernel exhibits sufficiently strong growth. Consequently, in the absence of fragmentation, we anticipate a similar asymptotic result for the growth-coagulation equations~\eqref{maineq.1}--\eqref{maineq.2}. The following theorem rigorously establishes this large-time behavior, providing the precise conditions under which the moments of the solution decay over time.
	\begin{theorem}\label{Largetimebehavior}
	Let $\Upsilon$ and $\mu$ satisfy \eqref{multiplicativecoagassum} and \eqref{mu}, respectively. Assume that $\alpha \equiv 0$ and $\beta \equiv 0$, and let $g$, $a$, and $\xi^{{\mathrm{in}}}$ be parameters satisfying \eqref{growthassum}--\eqref{initialassum}. Moreover, assume that $\xi^{\mathrm{in}} \not\equiv 0$ and $\xi$ is a weak solution to \eqref{maineq.1}--\eqref{maineq.2}, as defined in \Cref{definitionofsolution}. Additionally, suppose that  
		\begin{align}\label{deathdominated}  
			g(u) \leq u\mu(u), \quad u \in (0,\infty).
		\end{align}
	Then the following properties are satisfied. 
		\begin{enumerate}[label=(\alph*)]
			\item If for every $\theta>0$, there is $\delta_{\theta}>0$  such that
			\begin{align}\label{coaglowerbound}
				\Upsilon(u,u_{1})\geq \delta_{\theta}\quad \text{ for }\quad (u,u_{1})\in (\theta, \infty)^{2},
			\end{align}
			then
			\begin{align}\label{vanishingzerothmoment}
				\lim_{t \rightarrow \infty} M_{0}(\xi(t))=0.
			\end{align}
		\item If the coagulation kernel $\Upsilon$ satisfies 
		\begin{align}\label{coaglowerboundM}
			\Upsilon(u,u_{1})\geq \Upsilon_{2}(uu_{1})^{\lambda/2}, \quad (u, u_{1})\in (0, \infty)^{2},
		\end{align}
		for some $\lambda \in (1, 2]$ and $\Upsilon_{2}>0$, then
		\begin{align}\label{vanishingfirstmoment}
			\lim_{t \rightarrow \infty} M_{1}(\xi(t))=0.
		\end{align}
		\end{enumerate}
	\end{theorem} 
The proof of \Cref{Largetimebehavior} closely follows the approach used in \cite[Theorem 1.1]{escobedo2002gelation} and \cite[Theorem 2.5 and Proposition 4.1]{lachowicz2003oort}. Notably, the functions $g$ and $\mu$ considered in~\cite[Examples 7.1 and 7.3]{ackleh2023finite} satisfy condition \eqref{deathdominated}.
\section{existence of weak solutions}\label{existence section}
		\subsection{Truncated Problem}\label{Truncated Problem}
		The proof of \Cref{main theorem} relies on a weak compactness method within the functional space $L^{1}([0, \infty);(1+u)du)$. To employ a weak $L^{1}$-compactness technique, we have to generate a sequence of solutions to certain truncated problems. To formulate the truncated problem, we first approximate the initial data $\xi^{{\mathrm{in}}}$. For this purpose, we introduce the de la Vallée-Poussin theorem (see \cite[Proposition I.1.1]{Le1977} and \cite[Theorem 7.1.6]{BLL_book}), which guarantees the existence of a convex, non-negative, non-decreasing function $j_{0}\in C^{\infty}([0,\infty))$ satisfying $j_{0}(0)=j_{0}^{\prime}(0)=0$, such that $j_{0}^{\prime}$ is concave, $j_{0}^{\prime}(u)>0$ for $u>0$, and 
			\begin{align}\label{convexinfty}
				\lim_{u\rightarrow \infty}j_{0}^{\prime}(u)=\lim_{u\rightarrow \infty}\frac{j_{0}(u)}{u}=\infty,
			\end{align}
		along with the condition 
			\begin{align}\label{convexinitialdata}
				\int_{0}^{\infty}j_{0}(u)\xi^{{\mathrm{in}}}(u)\;du <\infty.
			\end{align}
Based on \cite[Section 4]{laurenccot2007well}, there exists a sequence $\left(\xi_{n}^{{\mathrm{in}}}\right)_{n \ge 1}$ of non-negative functions in $C_{c}^{\infty}(0,\infty)$ satisfying  
	\begin{align}\label{mollificationofinitialdata}  \xi_{n}^{{\mathrm{in}}} \to \xi^{{\mathrm{in}}} \quad {\mathrm{in}} \quad L^{1}([0, \infty))\quad \text{and} \quad C_{0} := \sup_{n \ge 1} \int_{0}^{\infty} j_{0}(u)\xi_{n}^{{\mathrm{in}}}(u)\;du < \infty.  
		\end{align}  
		Furthermore, assumptions \eqref{initialassum}, \eqref{convexinfty}, and \eqref{mollificationofinitialdata} ensure that  
		\begin{align}\label{initialdatauniformbound} 
			C_{1} := \sup_{n \ge 1} \int_{0}^{\infty} (1+u)\xi_{n}^{{\mathrm{in}}}(u)\;du < \infty.  
		\end{align}	 
Assumue that $\Upsilon$, $\beta$, $\alpha$, $\mu$, $g$, $a$, and $\xi^{{\mathrm{in}}}$ satisfy \eqref{multiplicativecoagassum}, \eqref{particlebound}, \eqref{lmassconservation}, \eqref{alphaassum}, \eqref{localbound}, \eqref{mu},  \eqref{growthassum}, \eqref{birthrateassum}, and \eqref{initialassum}. For $(t,u)\in (0,\infty)^2$ and $n\geq 1$, we consider the following truncated problem
		\begin{align}	
		\partial_{t}\xi_{n}(t,u) &= -\partial_{u}(g_{n}\xi_{n})(t,u)-\mu_{n}(u)\xi_{n}(t,u)+\mathcal{F}_{n}(\xi)(t,u)+\mathcal{C}_{n}(\xi_{n})(t,u), \label{truncated-1}\\
		\lim_{u\rightarrow 0^{+}}g_{n}(u)\xi_{n}(t,u)&=\int_{0}^{\infty}a_{n} (u_{1})\xi_{n}(t,u_{1})\;du_{1},\quad \xi_{n}(0,u)=\xi_{n}^{{\mathrm{in}}}(u) \geq 0, \label{truncated-2}  
	\end{align}
		where
		\begin{align}\label{trungmubetaF}
			g_{n}:=g+\frac{1}{n},\quad \mu_{n}:=\mu\chi_{[0,n]},\quad \alpha_{n}:=\alpha\chi_{[0,n]},\quad \beta_{n}:=\beta,\quad a_{n}:=a,
		\end{align}
		\begin{align}
			\mathcal{F}_{n}(\xi_{n})(t,u)&:=-\alpha_{n}(u)\xi_{n}(t,u)+\int_{u}^{\infty} \alpha_{n}(u_{1})\beta(u|u_{1})\xi_{n}(t,u_{1})\;du_{1},\nonumber \\
		\mathcal{C}_{n}(\xi_{n})(t,u)&:=\frac{1}{2} \int_{0}^{u} \Upsilon_{n}(u-u_{1},u_{1})\xi_{n}(t,u-u_{1})\xi_{n}(t,u_{1})\;du_{1} \nonumber \\
			&\quad	-\int_{0}^{\infty} \Upsilon_{n}(u,u_{1})\xi_n(t,u)\xi_n(t,u_{1})\;du_{1},\nonumber 
		\end{align}
	and
		\begin{align}\label{truncoag}
			\Upsilon_n(u,u_{1}):&=\Upsilon(u,u_{1}) \chi_{[0,n]}(u) \chi_{[0,n]}(u_{1}),\quad (u,u_{1})\in (0,\infty)^{2}.
		\end{align}
	The function $\chi_{S}$ represents the indicator function of the set $S$. It is given by 
	\begin{align*}
		\chi_{S}(u):=\begin{cases}
			1\quad u\in S,\\
			0\quad u\notin S.
		\end{cases}
	\end{align*}
Clearly, for $n \geq 1$, the function $g_n$ satisfies the following conditions  
\begin{align}\label{gnbound}  
	\int_{0^{+}}\frac{du}{g_{n}(u)}<\infty\quad \text{and} \quad \|g_{n}\|_{\infty}\leq \|g\|_{\infty}+1.
\end{align}  
Furthermore, $\mu_{n} \in L^{\infty}(0,\infty)$ and $\Upsilon_n \in L^{\infty}((0,\infty)^{2})$. For $n\ge 1$ and $T>0$, the truncated version of weak formulation is given by 
		\begin{align} \label{truncated weak formulation}
			\int_{0}^{\infty} \xi_{n}(t,u)\vartheta(u)\;du  &=   	\int_{0}^{\infty} \xi_{n}^{\mathrm{in}}(u)\vartheta(u)\;du   
			+\int_{0}^{t}\int_{0}^{\infty} \vartheta^{\prime}(u)g_{n}(u)\xi_{n}(s,u)\;duds\nonumber \\ &+\vartheta(0)\int_{0}^{t}\int_{0}^{\infty} a_{n}(u)\xi_{n}(s,u)\;duds-\int_{0}^{t}\int_{0}^{\infty}\vartheta(u) \mu_{n}(u)\xi_{n}(s,u)\;duds \nonumber \\
			&+\int_{0}^{t}\int_{0}^{\infty}\psi(\vartheta)(u_{1}) \alpha_{n}(u_{1})\xi_{n}(s,u_{1})\;du_{1}ds  \nonumber \\
			&+\frac{1}{2}\int_{0}^{t}\int_{0}^{\infty} \int_{0}^{\infty}\tilde{\vartheta}(u,u_{1}) \Upsilon_{n}(u,u_{1})\xi_{n}(s,u) \xi_{n}(s,u_{1})\;dudu_{1}ds,
		\end{align}
for each $t\in [0, T]$ and $\vartheta \in C[0,\infty)\cap W^{1,\infty}([0,\infty))$. Owing to the properties enjoyed by $g_{n}$, $\mu_{n}$, $\alpha_{n}$, $\beta_{n}$, $\Upsilon_{n}$, and $a_{n}$,  we are now prepared to recall the following result from \cite[Theorem 3.5]{banasiak2009coagulation}.
	\begin{proposition}\label{ClassicalandMildsolution}
		For each $n\ge 1$, the truncated problem \eqref{truncated-1}--\eqref{truncated-2} has a unique, global, and strongly differentiable classical solution $\xi_{n}(t)\in \mathcal{Z}_{n}\cap Y_{+}$ for each $t\in [0,\infty)$ that satisfies the following property
		\begin{align}\label{tempmomentbound}
			\int_{0}^{\infty}\xi_{n}(t,u)(1+u)\;du\leq \left(\int_{0}^{\infty}\xi_{n}^{\mathrm{in}}(u)(1+u)\;du+\frac{C_{2}^{n}}{C_{3}^{n}}\right)e^{C_{3}^{n}t}-\frac{C_{2}^{n}}{C_{3}^{n}}\quad \text{for}\quad t\in [0,\infty),
		\end{align}
	where
	\begin{align}
		&\mathcal{Z}_{n}:=\left\{ f\in 	\mathcal{Z}_{n}^{1}\; |\;g_{n}f\in AC((0,\infty)),\; \mathcal{T}_{n}f\in \mathcal{G} \right\},\nonumber\\
		&\mathcal{Z}_{n}^{1}:=\left\{f\in Y \; |\; 	[\mathcal{B}_{n}f_{+}](u)<\infty,\; 	[\mathcal{B}_{n}f_{-}](u)<\infty\;\text{ a.e.} \right\},\nonumber\\
		&[\mathcal{B}_{n}f](u):=\int_{u}^{\infty} \alpha_{n}(u_{1})\beta_{n}(u|u_{1})f(u_{1})\;du_{1},\;f_{+}:=\sup\left\{f,0\right\},\;f_{-}:=\sup\left\{-f,0\right\}, \nonumber\\ 
		&\mathcal{T}_{n}f(u):=-\partial_{u}(g_{n}f)(u)-\mu_{n}(u)f(u) -\alpha_{n}(u)f(u),\; u\in (0,\infty), \nonumber\\
		&\mathcal{G}:=\left\{f\;|\; f: [0,\infty)\mapsto\mathbb{R}\cup \left\{-\infty,\infty\right\},\; f~\text{is~ measurable~ and finite almost everywhere} \right\},\nonumber\\
		&C_{3}^{n}:=\|a_{n}\|_{\infty}+\|g_{n}\|_{\infty}+M\max\left\{P_{\alpha},Q_{\alpha}\right\},\nonumber
	\end{align}
and
\begin{align}\label{C_{2}}
	C_{2}^{n}:=\lim_{m\rightarrow \infty}\int_{N_{m}}^{\infty}\alpha_{n}(u_{1})\xi_{n}^{{\mathrm{in}}}(u_{1})\left(\int_{0}^{N_{m}}\beta(u|u_{1})(1+u)\;du\right)\;du_{1}.
\end{align}
Here, $(N_{m})_{m\ge 1}$ is a sequence such that $N_{m}\rightarrow \infty$ as $m\rightarrow \infty$. Thanks to \eqref{localbound} and the compact support of the initial data $\xi_{n}^{{\mathrm{in}}}$, the limit in \eqref{C_{2}} exists, and we have $C_{2}^{n}=0$. Therefore, using \eqref{initialdatauniformbound} and $C_{2}^{n}=0$ in \eqref{tempmomentbound}, we obtain
	\begin{align}\label{momentbound}
	\int_{0}^{\infty}\xi_{n}(t,u)(1+u)\;du\leq C_{1}e^{C_{3}t}\quad \text{for}\quad t\in [0,\infty),\;\; n\ge 1,
\end{align}
where $C_{3}:=\|a\|_{\infty}+\|g\|_{\infty}+M\max\left\{P_{\alpha},Q_{\alpha}\right\}+1$. Moreover, for $n\ge 1$ and $t\in [0, \infty)$, $\xi_{n}$ satisfies the following mild formulation
\begin{align}\label{mildformulation}
	\xi_{n}(t,u)=S_{n}(t)\xi_{n}^{{\mathrm{in}}}(u)+\int_{0}^{t}S_{n}(t-s)\mathcal{C}_{n}(\xi_{n})(s,u)\;ds,
\end{align}
where $\left\{S_{n}(t)\right\}_{t\geq 0}$ is a positive semigroup associated with the problem \eqref{truncated-1}--\eqref{truncated-2} on $Y$ with
\begin{align}\label{Snormbound}
	\|S_{n}(t)\|_{Y}\leq e^{\tilde{b}_{n}t},\quad t\ge 0.
\end{align}
Here, $\|\cdot\|_{Y}$ denotes the operator norm on $Y$ and
\begin{align}
	\tilde{b}_{n}&:=\left(\|a_{n}\|_{\infty}+\|g_{n}\|_{\infty}+(M+1)\|\alpha_{n}\|_{\infty}\right)\leq\tilde{b},\quad n\ge 1,\label{tildebnbound}\\
	 \tilde{b}&:=\left(\|a\|_{\infty}+\|g\|_{\infty}+(M+1)\|\alpha\|_{\infty}+1\right)\nonumber.
\end{align}
\end{proposition}
The following lemma follows directly from \Cref{ClassicalandMildsolution}.
\begin{lemma}\label{TruncatedWeakExistence}
	For each $n\ge 1$ and $T>0$, the unique global, differentiable classical solution to the truncated problem~\eqref{truncated-1}--\eqref{truncated-2} satisfies \eqref{truncated weak formulation}, and the following holds
	\begin{align}\label{Massconservation}
		\int_{0}^{\infty}u\xi_{n}(t,u)\;du=\int_{0}^{\infty}u\xi_{n}^{{\mathrm{in}}}(u)\;du+ \int_{0}^{t}\int_{0}^{\infty} \left(g_{n}(u)-u\mu_{n}(u)\right)\xi_{n}(s,u)\;duds,
	\end{align}
for each $t\in [0,T]$.
\end{lemma}
\begin{proof}
	Clearly, for each $n\ge 1$ and $T>0$, the unique classical solution to \eqref{truncated-1}--\eqref{truncated-2} satisfies \eqref{truncated weak formulation} for each $t\in [0, T]$ and $\vartheta \in C[0,\infty)\cap W^{1,\infty}([0,\infty))$,  since every differentiable classical solution to~\eqref{truncated-1}--\eqref{truncated-2} satisfies the weak formulation~\eqref{truncated weak formulation}. Next, to show \eqref{Massconservation}, we consider $t\in [0, T]$ and $\vartheta_{R}\in C_{c}^{1}([0,\infty))$ with 
	\begin{align*}
		\vartheta_{R}(u)=u~\text{ for }~0\leq u \leq R,~\vartheta_{R}(u)=0~\text{ for }~u\ge 2R,
	\end{align*}
and 
	 \begin{align*}
	 	0\leq \vartheta_{R}(u)\leq 2u,\quad |\vartheta_{R}^{\prime}(u)|\leq 2,~\text{ for }~u\in [0,\infty).
	 \end{align*}
	Plugging $\vartheta_{R}$ into~\eqref{truncated weak formulation}, we obtain 
	\begin{align}
		\int_{0}^{\infty} \xi_{n}(t,u)\vartheta_{R}(u)\;du  &=   	\int_{0}^{\infty} \xi_{n}^{{\mathrm{in}}}(u)\vartheta_{R}(u)\;du   
		+\int_{0}^{t}\int_{0}^{\infty} \vartheta_{R}^{\prime}(u)g_{n}(u)\xi_{n}(s,u)\;duds\nonumber \\ &-\int_{0}^{t}\int_{0}^{\infty}\vartheta_{R}(u) \mu_{n}(u)\xi_{n}(s,u)\;duds +\int_{0}^{t}\int_{0}^{\infty}\psi(\vartheta_{R})(u_{1}) \alpha_{n}(u_{1})\xi_{n}(s,u_{1})\;du_{1}ds  \nonumber \\
		&+\frac{1}{2}\int_{0}^{t}\int_{0}^{\infty} \int_{0}^{\infty}\widetilde{\vartheta_{R}}(u,u_{1}) \Upsilon_{n}(u,u_{1})\xi_{n}(s,u) \xi_{n}(s,u_{1})\;dudu_{1}ds,\label{tempmassconservationlemma}
	\end{align}
	where 
	\begin{align*}
		\psi(\vartheta_{R})(u_{1})=\int_{0}^{u_{1}}\vartheta_{R}(u)\beta(u|u_{1})\;du-\vartheta_{R}(u_{1}).
	\end{align*}
	Since 
		\begin{align*}
			&\xi_{n}(t)\in Y_{+},\; t\in [0,T],\quad \xi_{n}^{{\mathrm{in}}}\in Y_{+},\quad \|\alpha_{n}\|_{L^{\infty}(0,\infty)} \leq P_{\alpha}n+Q_{\alpha},\\
			 &\|\Upsilon_{n}\|_{L^{\infty}(0,\infty)} \leq \Upsilon_{0}(1+n)^{2},\; \|\mu_{n}\|_{L^{\infty}(0,\infty)} \leq P_{\mu}n+Q_{\mu},
		\end{align*}
	 with  
	\begin{equation*}  
		\widetilde{\eta_{R}}(u,u_{1}) = 0, \quad u+u_{1} \in (0,R), \qquad \big|\widetilde{\eta_{R}}(u,u_{1})\big| \leq 4 (u+u_{1}), \quad (u,u_{1}) \in (0,\infty)^2,  
	\end{equation*}  
	and, by \eqref{lmassconservation},  
	\begin{equation*}  
		\psi(\vartheta_{R})(u_{1}) = 0 \quad \text{for} \quad 0 \leq u_{1} \leq R,\quad |\psi(\vartheta_{R})(u_{1})|\leq 4u_{1}, \; u_{1}\in (0,\infty),  
	\end{equation*}  
	we establish \eqref{Massconservation} by passing to the limit as $R \to \infty$ in \eqref{tempmassconservationlemma}, utilizing \eqref{growthassum}, \eqref{momentbound}, and Lebesgue’s Dominated Convergence Theorem (DCT).  	 
\end{proof}
\subsection{Weak Compactness} \label{Global existence of weak solutions}
We begin by analyzing the tail behavior of the sequence of solutions $\left(\xi_{n}\right)_{n\ge 1}$ to equations \eqref{truncated-1}--\eqref{truncated-2}, utilizing estimates on super-linear moments in the following lemma. 
	\begin{lemma} \label{Tail-1}
		Suppose all the assumptions of \Cref{main theorem}$(a)$ hold. If $\xi_{n}$ is a weak solution to \eqref{truncated-1}--\eqref{truncated-2} for each $n\ge 1$ and $T\in (0,\infty)$, then there exists $L(T)>0$ that depends only on $C_{0}$, $C_{1}$, $C_{3}$, $\|g\|_{\infty}$, $\Upsilon_{0}$, $\Upsilon_{1}$, $j_{0}$, and $T$ such that
		\begin{align}
			\int_{0}^{\infty} j_{0}(u)\xi_{n}(t,u) \;du \leq L(T), \label{j0}
		\end{align}
	for each $t\in [0,T]$.
Furthermore, for every $\epsilon \in (0,1)$, there is $R(\epsilon)\geq 1$ such that 
	\begin{align} \label{tailbehaviour}
		\sup_{n\geq 1} \sup_{t\in [0,T]}	\int_{R(\epsilon)}^{\infty} u\xi_{n}(t,u)\;du  \leq \epsilon.  
	\end{align}	
	\end{lemma}
	\begin{proof}
	In the entire proof, $L(T)$ represents a positive constant that depends only on $C_{0}$, $C_{1}$, $C_{3}$, $\|g\|_{\infty}$, $\Upsilon_{0}$, $\Upsilon_{1}$, $j_{0}$, and $T$. Fix $n\in \mathbb{N}$. For $R>n$, we define $j_{R}(u):=\min(j_{0}(u),j_{0}(R))$,  $u\in [0,\infty)$. Then $j_{R}$ belongs to $C[0,\infty)\cap W^{1,\infty}([0,\infty))$, and the following inequalities hold, as established in \cite[Lemma A.2]{Laurencot2001}, \cite[Proposition 2.14]{laurenccot2014weak}, and \cite[Lemma 3.3]{Giri2021weak}, for $(u,u_{1})\in (0,\infty)^2$ 
		\begin{align}
			(u+u_{1})\left(j_{R}(u+u_{1})-j_{R}(u)-j_{R}(u_{1})\right)&\leq 2\left(u_{1}j_{R}(u)+uj_{R}(u_{1})\right),  \label{convex inequality-1}\\
			0\leq (1+u)j_{R}^{\prime}(u)&\leq 3j_{R}(u)+j_{0}(1), \label{convex inequality-2}\\
			j_{R}(u+u_{1})-j_{R}(u)-j_{R}(u_{1})&\leq j_{0}^{\prime\prime}(0)uu_{1}, \label{Mconvex inequality-1}\\
				j_{R}(u+u_{1})-j_{R}(u)-j_{R}(u_{1})&\leq 4u\frac{j_{R}(u_{1})}{u_{1}},\;0<u\leq u_{1},  \label{Mconvex inequality-2}
		\end{align}
	and 
	\begin{align}
		\int_{0}^{u_{1}}j_{R}(u)\beta(u|u_{1})\;du &\leq j_{R}(u_{1}),\;u_{1}\in (0,R). \label{convex inequality-3}
	\end{align}
	Next, by choosing $\vartheta =j_{R}$ in \eqref{truncated weak formulation}, we obtain
			\begin{align} 
			\int_{0}^{\infty} j_{R}(u)\xi_{n}(t,u)\;du  &=   	\int_{0}^{\infty} j_{R}(u)\xi_{n}^{{\mathrm{in}}}(u)\;du   
			+\int_{0}^{t}\int_{0}^{\infty} j_{R}^{\prime}(u)g_{n}(u)\xi_{n}(s,u)\;duds\nonumber \\ &+j_{R}(0)\int_{0}^{t}\int_{0}^{\infty} a_{n}(u)\xi_{n}(s,u)\;duds-\int_{0}^{t}\int_{0}^{\infty} \mu_{n}(u)j_{R}(u)\xi_{n}(s,u)\;duds \nonumber \\
			&+\int_{0}^{t}\int_{0}^{\infty}\left(\int_{0}^{u_{1}}j_{R}(u)\beta(u|u_{1})\;du-j_{R}(u_{1})\right) \alpha_{n}(u_{1})\xi_{n}(s,u_{1})\;du_{1}ds \nonumber \\
			&+\frac{1}{2}\int_{0}^{t}\int_{0}^{\infty} \int_{0}^{\infty}\tilde{j_{R}}(u,u_{1}) \Upsilon_{n}(u,u_{1})\xi_{n}(s,u) \xi_{n}(s,u_{1})\;dudu_{1}ds,\nonumber
		\end{align}
		for every $t\in [0,T]$, where
		\begin{align}
			\widetilde{j_R}(u,u_{1}):=j_{R}(u+u_{1})-j_{R}(u)-j_{R}(u_{1}). \nonumber 
		\end{align}
		Using \eqref{mollificationofinitialdata} and \eqref{convex inequality-3}, along with the facts that the support of $\alpha_{n}$ is a subset of $[0,R]$, $j_{R}(0)=0$, and that $\xi_{n}$, $\mu_{n}$, and $j_{R}$ are non-negative, we deduce the following inequality
		\begin{align}
			\int_{0}^{\infty} j_{R}(u)\xi_n(t,u) \;du	&\leq C_{0}+\int_{0}^{t}\int_{0}^{\infty} j_{R}^{\prime}(u)g_{n}(u)\xi_{n}(s,u)\;duds\nonumber\\
			&\quad +\frac{1}{2}\int_{0}^{t}\int_{0}^{\infty} \int_{0}^{\infty} \widetilde{j_R}(u,u_{1}) \Upsilon_{n}(u,u_{1})\xi_{n}(s,u) \xi_{n}(s,u_{1}) \;dudu_{1} ds.\nonumber
	 \end{align}
By utilizing~\eqref{Mconvex inequality-1}, \eqref{Mconvex inequality-2}, \eqref{convex inequality-1}, \eqref{gnbound}, \eqref{convex inequality-2}, \eqref{multiplicativecoagassum}, \eqref{additivecoagassum}, and \eqref{momentbound} in the above inequality, we derive the following estimate
	\begin{align}
			\int_{0}^{\infty} j_{R}(u)\xi_{n}(t,u) \;du	&\leq C_{0}+\|g_{n}\|_{\infty}\int_{0}^{t}\int_{0}^{\infty}(1+u) j_{R}^{\prime}(u)\xi_{n}(s,u)\;duds\nonumber\\
			&\quad +\frac{j_{0}^{\prime\prime}(0)}{2}\int_{0}^{t}\int_{0}^{1} \int_{0}^{1}  \Upsilon_{n}(u,u_{1})\xi_{n}(s,u) \xi_{n}(s,u_{1}) \;dudu_{1} ds\nonumber\\
			&\quad +4\int_{0}^{t}\int_{1}^{\infty} \int_{0}^{1} \frac{u}{u_{1}} j_{R}(u_{1}) \Upsilon_{n}(u,u_{1})\xi_{n}(s,u) \xi_{n}(s,u_{1}) \;dudu_{1} ds\nonumber\\
			&\quad +\int_{0}^{t}\int_{1}^{\infty} \int_{1}^{\infty} \left( \frac{u_{1}j_{R}(u)+uj_{R}(u_{1})}{u+u_{1}}\right) \Upsilon_{n}(u,u_{1})\xi_{n}(s,u) \xi_{n}(s,u_{1}) \;dudu_{1} ds\nonumber\\
			&\leq C_{0}+\left(1+\|g\|_{\infty}\right)\int_{0}^{t}\int_{0}^{\infty}\left(3j_{R}(u)+j_{0}(1)\right) \xi_{n}(s,u)\;duds +\frac{j_{0}^{\prime\prime}(0)}{2}\Upsilon_{0}TC_{1}^{2}e^{2C_{3}T}
			\nonumber\\
			&\quad + 8\Upsilon_{0}C_{1}e^{C_{3}T}\int_{0}^{t}\int_{0}^{\infty}j_{R}(u_{1}) \xi_{n}(s,u_{1})\;du_{1}ds\nonumber\\
			&\quad +\Upsilon_{1}\int_{0}^{t}\int_{1}^{\infty} \int_{1}^{\infty} \left( u_{1}j_{R}(u)+uj_{R}(u_{1})\right) \xi_{n}(s,u) \xi_{n}(s,u_{1}) \;dudu_{1} ds\nonumber\\
			&\leq C_{0}+\left(1+\|g\|_{\infty}\right)j_{0}(1)TC_{1}e^{C_{3}T}+\frac{j_{0}^{\prime\prime}(0)}{2}\Upsilon_{0}TC_{1}^{2}e^{2C_{3}T}\nonumber\\
			&\quad+\left[3\left(1+\|g\|_{\infty}\right)+8\Upsilon_{0}C_{1}e^{C_{3}T}+2\Upsilon_{1}C_{1}e^{C_{3}T}\right]\int_{0}^{t}\int_{0}^{\infty}j_{R}(u) \xi_{n}(s,u)\;duds.\nonumber
	\end{align}
Applying Gronwall’s inequality, we derive the desired estimate
		\begin{align}
			\int_{0}^{\infty} j_{R}(u)\xi_{n}(t,u)\;du  \leq L(T),\quad t \in [0, T],\;n\ge 1, \nonumber 
		\end{align}
	where 
		\begin{align}
	 L(T):=\left(C_{0}+\frac{j_{0}^{\prime\prime}(0)}{2}\Upsilon_{0}TC_{1}^{2}e^{2C_{3}T}+\left(1+\|g\|_{\infty}\right)j_{0}(1)TC_{1}e^{C_{3}T}\right)e^{L_{K}(T)}, \nonumber 
	\end{align}
and $L_{K}(T):=T\left[3\left(1+\|g\|_{\infty}\right)+8\Upsilon_{0}C_{1}e^{C_{3}T}+2\Upsilon_{1}C_{1}e^{C_{3}T}\right]$. Then~\eqref{j0} follows from Fatou's lemma by taking the limit as $R\rightarrow \infty$. Next, we proceed to establish \eqref{tailbehaviour} by considering $R_{0}\geq 1$ and deriving 
\begin{align}
	\sup_{n\geq 1} \sup_{t\in [0,T]}	\int_{R_{0}}^{\infty} u\xi_n(t,u)\;du \leq \sup_{n\geq 1} \sup_{t\in [0,T]} \sup_{u\in [R_{0},\infty)}\left(\frac{u}{j_{0}(u)}\right)	\int_{R_{0}}^{\infty} j_{0}(u)\xi_n(t,u)\;du. \nonumber
\end{align}
Utilizing \eqref{j0}, we obtain
\begin{align}	
	\sup_{n\geq 1} \sup_{t\in [0,T]}	\int_{R_{0}}^{\infty} u\xi_n(t,u)\;du \leq L(T) \sup_{u\in [R_{0},\infty)}\left(\frac{u}{j_{0}(u)}\right).  \nonumber
\end{align}	
The above result, combined with \eqref{convexinfty}, establishes \eqref{tailbehaviour}.
\end{proof}	
\begin{corollary}\label{second moment lemma-1}
		Assume that all hypotheses of \Cref{main theorem}$(a)$ are true and that $\xi^{\mathrm{in}}$ belongs to $L^1\left([0,\infty);u^2 du\right)$. Then, for every $T>0$ and $n\geq 1$, there exists $L(T)>0$ that depends only on $C_{0}$, $C_{1}$, $C_{3}$, $\|g\|_{\infty}$, $\Upsilon_{0}$, $\Upsilon_{1}$, and $T$ such that
		\begin{align}
			M_{2}(\xi_{n}(t)) \leq L(T),
		\end{align}
	for every $t \in [0, T]$ and $n\ge 1$.
	\end{corollary}
	\begin{proof}
		Observe that $j_{2} : u \mapsto  u^2$
		is a $C^{\infty}$ convex function on $[0,\infty)$ with concave derivative, which satisfies
		$j_{2}(0)=j_{2}^{\prime}(0) = 0$. Hence,  \Cref{second moment lemma-1}
		immediately follows from \Cref{Tail-1} and \eqref{convexinitialdata}--\eqref{mollificationofinitialdata} with  $j_{0}=j_{2}$.
	\end{proof}
The findings of the preceding section establish that matter cannot escape to infinity. The objective of the subsequent result is to demonstrate the uniform integrability of the sequence $\left(\xi_{n}\right)_{n\ge 1}$, thereby ruling out the possibility of matter concentrating at a finite size. 
	\begin{proposition} \label{uniform integrability proposition}
Assume that $\Upsilon$, $\beta$, $\alpha$, $\mu$, $g$, $a$, and $\xi^{{\mathrm{in}}}$ satisfy \eqref{multiplicativecoagassum}, \eqref{particlebound}, \eqref{lmassconservation}, \eqref{alphaassum}, \eqref{localbound}, \eqref{mu},  \eqref{growthassum}, \eqref{birthrateassum}, and \eqref{initialassum}. Then, given any $\epsilon > 0$ and $T>0$, there exists a $\delta(\epsilon) > 0$ such that
		\begin{align}
			\sup_{n\geq 1}\sup_{t\in [0,T]}\int_{\mathcal{P}}  \xi_n(t,u) \;du \leq \epsilon,\; \text{whenever}\;\;|\mathcal{P}| \leq \delta(\epsilon),\nonumber
		\end{align}
	where $\mathcal{P}$ is any Lebesgue measurable subset of $[0,\infty)$ and $|\mathcal{P}|$ denotes the standard Lebesgue measure of $\mathcal{P}$.
	\end{proposition}
	\begin{proof}
	We follow a similar approach as presented in \cite[Lemma 3.5]{Stewart1989}, \cite[Lemma 3.5]{Giri2021weak}, and \cite[Proposition 4.4]{leis2017existence} to prove \Cref{uniform integrability proposition}. Set $T>0$. Then, for $t\in [0,T]$, $\delta \in (0,1)$, $R>1$, and $n\geq 1$, we define the quantities
		 \begin{align}
			\mathcal{E}_{n,R}(t,\delta):=	\sup  \left\{	\int_{\mathcal{P}} \xi_{n}(t,u)\;du\;\mid\;\mathcal{P}\subset [0,R), \; |\mathcal{P}|\leq \delta \right\}  \nonumber
		\end{align}
and
	 \begin{align}
	 	\mathcal{E}_{n}^{0}(\delta):=	\sup  \left\{	\int_{\mathcal{P}} \xi_{n}^{{\mathrm{in}}}(u)\;du\;\mid\;\mathcal{P}\subset [0,\infty), \; |\mathcal{P}|\leq \delta \right\}.  \nonumber
	 \end{align}
We observe that
 \begin{align}
 		\mathcal{E}_{n,R}(0,\delta)\leq \mathcal{E}_{n}^{0}(\delta).
 \end{align}
 Moreover, from \eqref{mollificationofinitialdata} and the Dunford-Pettis theorem \cite[Theorem 4.21.2]{edwards2012functional}, we obtain the uniform convergence property
\begin{align}\label{uniforminitial1}
	\lim_{\delta\rightarrow 0}\sup_{n \ge 1}\mathcal{E}_{n}^{0}(\delta)=0.
\end{align}
Next, we consider a measurable subset $\mathcal{P}\subset [0,R)$ having $|\mathcal{P}|\leq \delta$. From \eqref{mildformulation}, \eqref{Snormbound}, \eqref{tildebnbound}, and Fubini's theorem, we deduce the following inequality 
		\begin{align}
		\int_{\mathcal{P}}\xi_{n}(t,u) \;du &=\int_{\mathcal{P}} S_{n}(t)\xi_{n}^{{\mathrm{in}}}(u) \;du+\int_{0}^{t}	\int_{\mathcal{P}} S_{n}(t-s)\mathcal{C}_{n}(\xi_{n})(s,u)\;duds\nonumber\\
		&\leq \int_{\mathcal{P}} 	\|S_{n}(t)\|_{Y}\xi_{n}^{{\mathrm{in}}}(u) \;du+\int_{0}^{t}	\int_{\mathcal{P}} 	\|S_{n}(t-s)\|_{Y}\left|\mathcal{C}_{n}(\xi_{n})(s,u)\right|\;duds\nonumber\\
		&\leq e^{\tilde{b}_{n}t}\int_{\mathcal{P}} 	\xi_{n}^{{\mathrm{in}}}(u) \;du+\int_{0}^{t}e^{\tilde{b}_{n}(t-s)}	\int_{\mathcal{P}} \left|\mathcal{C}_{n}(\xi_{n})(s,u)\right|\;duds\nonumber\\
		&\leq e^{\tilde{b}t}\int_{\mathcal{P}} \xi_{n}^{{\mathrm{in}}}(u) \;du\nonumber\\
		&\quad +\frac{e^{\tilde{b}t}}{2}\int_{0}^{t}	\int_{0}^{\infty}\int_{0}^{\infty}\left[\chi_{\mathcal{P}}(u+u_{1})+\chi_{\mathcal{P}}(u)+\chi_{\mathcal{P}}(u_{1})\right]\Upsilon_{n}(u,u_{1})\xi_{n}(s,u)\xi_{n}(s,u_{1})\;du_{1}du ds,\label{uniformcoagterm1}			
		\end{align}
	for each $t\in [0,T]$. Let us compute the last term on the right-hand side of \eqref{uniformcoagterm1}. By applying \eqref{multiplicativecoagassum}, we obtain
		\begin{align}
	 &\int_{0}^{\infty}\int_{0}^{\infty}\left[\chi_{\mathcal{P}}(u+u_{1})+\chi_{\mathcal{P}}(u)+\chi_{\mathcal{P}}(u_{1})\right]\Upsilon_{n}(u,u_{1})\xi_{n}(s,u)\xi_{n}(s,u_{1})\;du_{1}du\nonumber\\
	 &\leq \Upsilon_{0}\int_{0}^{\infty}\int_{0}^{\infty}\left[\chi_{-u+\mathcal{P}}(u_{1})+\chi_{\mathcal{P}}(u)+\chi_{\mathcal{P}}(u_{1})\right](1+u)(1+u_{1})\xi_{n}(s,u)\xi_{n}(s,u_{1})\;du_{1}du.\nonumber
	\end{align}
Using the inclusion $-u + \mathcal{P} \cap [0, \infty) \subset [0, R)$ for $u\in (0,\infty)$, the translation invariance of the Lebesgue measure, and \eqref{momentbound}, we obtain the estimate
  \begin{align}
   &\int_{0}^{\infty}\int_{0}^{\infty}\left[\chi_{\mathcal{P}}(u+u_{1})+\chi_{\mathcal{P}}(u)+\chi_{\mathcal{P}}(u_{1})\right]\Upsilon_{n}(u,u_{1})\xi_{n}(s,u)\xi_{n}(s,u_{1})\;du_{1}du\nonumber\\
  &\leq  3\Upsilon_{0}(1+R) C_{1}e^{C_{3}s}\mathcal{E}_{n,R}(s,\delta),\quad s\in (0,t),\label{uniformcoagterm2}	
 \end{align}
for each $t\in [0,T]$. Substituting \eqref{uniformcoagterm2} into \eqref{uniformcoagterm1}, we conclude that
\begin{align}
\mathcal{E}_{n,R}(t,\delta)\leq e^{\tilde{b}T}\mathcal{E}_{n}^{0}(\delta)+\frac{3}{2}\Upsilon_{0}(1+R)C_{1}e^{\left(C_{3}+\tilde{b}\right)T}\int_{0}^{t}\mathcal{E}_{n,R}(s,\delta)\;ds,\nonumber
\end{align}
for every $t\in [0,T]$. Then, applying Gronwall's inequality, we obtain
		\begin{align}
			\mathcal{E}_{n,R}(t,\delta) &\leq e^{\tilde{b}T}e^{\frac{3}{2}T\Upsilon_{0}(1+R) C_{1}e^{\left(C_{3}+\tilde{b}\right)T}} \mathcal{E}_{n}^{0}(\delta). \label{uniformcoagterm3}  
		\end{align}
Next, for $\delta \in (0,1)$ and $T>0$, we introduce the following notation
	\begin{align}
		\mathcal{E}(\delta):=	\sup  \left\{	\int_{\mathcal{P}} \xi_{n}(t,u)\;du\; |\;\mathcal{P}\subset [0,\infty), \; |\mathcal{P}|\leq \delta, \;t\in [0,T], \;n\geq 1 \right\},  \nonumber
	\end{align}
where $\mathcal{P}$ is a measurable subset of $[0,\infty)$ with $|\mathcal{P}|\leq \delta$. For $n\geq 1$, $T>0$, $t\in [0,T]$, $\delta \in (0,1)$, and $R>1$, from \eqref{momentbound} and \eqref{uniformcoagterm3}, we infer that
\begin{align}
	\int_{\mathcal{P}}\xi_{n}(t,u)\;du&\leq 	\int_{\mathcal{P}\cap [0,R)}\xi_{n}(t,u)\;du+\int_{R}^{\infty}\xi_{n}(t,u)\;du\nonumber\\
	&\leq \mathcal{E}_{n,R}(t,\delta)+\frac{1}{R} \int_{R}^{\infty}u\xi_{n}(t,u)\;du\nonumber\\
	&\leq e^{\tilde{b}T}e^{\frac{3}{2}T\Upsilon_{0}(1+R) C_{1}e^{\left(C_{3}+\tilde{b}\right)T}} \mathcal{E}_{n}^{0}(\delta)+\frac{1}{R}C_{1}e^{C_{3}T}.\nonumber 
\end{align}
By taking the supremum over $n$, $t$, and $\mathcal{P}$ on both sides, we derive the following inequality
\begin{align}
\mathcal{E}(\delta)
	\leq e^{\tilde{b}T}e^{\frac{3}{2}T\Upsilon_{0}(1+R) C_{1}e^{\left(C_{3}+\tilde{b}\right)T}} \sup_{n \ge 1}\mathcal{E}_{n}^{0}(\delta)+\frac{1}{R}C_{1}e^{C_{3}T}.\nonumber 
\end{align}
Utilizing \eqref{uniforminitial1}, we take the limit as $\delta\rightarrow 0$ in the previous inequality, yielding
\begin{align}
		\lim_{\delta\rightarrow 0} \mathcal{E}(\delta)\leq \frac{1}{R}C_{1}e^{C_{3}T}.\nonumber
\end{align}
Since the above inequality holds for any $R>1$ and $\mathcal{E}(\delta)$ is non-negative, we can take the limit as $R\rightarrow\infty$ to conclude that
\begin{align}
	\lim_{\delta\rightarrow 0}\mathcal{E}(\delta) = 0,\nonumber
\end{align}
thus completing the proof of \Cref{uniform integrability proposition}.
\end{proof}
For each $t\in [0,T]$, \Cref{uniform integrability proposition} ensures that the sequence $(\xi_n(t))_{n\geq 1}$ is uniformly integrable in $L^{1}([0, \infty))$, while \eqref{momentbound} establishes that it remains uniformly bounded in $Y_{+}$. Consequently, by applying the Dunford-Pettis theorem \cite[Theorem 4.21.2]{edwards2012functional}, we conclude that  
\begin{align}  
	\begin{cases}  
		&(\xi_n(t))_{n\geq 1} \text{ is weakly compact in } L^1([0,\infty)), \\  
		&\text{for each } t \in [0,T].  
	\end{cases} \label{weakly compact}  
\end{align}  
%%%%%%%%%%%%%%%%%%%%%%%
%%%%%%%%%%%%%%%%%%%%%%

\subsection {Equicontinuity in Time}
The analysis in the preceding sections resolves the weak compactness issue concerning the size variable $u$. To establish the relative compactness of the sequence $(\xi_n)_{n\geq 1}$ in $C([0,T];w\text{-}L^1([0,\infty)))$ for any $T>0$, we employ a generalized form of the Arzelà--Ascoli theorem \cite[Theorem 1.3.2]{Vrabie2003}. This requires demonstrating that the sequence $(\xi_n)_{n\geq 1}$ satisfies the following criterion.
	\begin{lemma} \label{time equicontinuity}
	Let the assumptions of \Cref{uniform integrability proposition} be satisfied. Then, for any $\vartheta \in L^{\infty}([0,\infty))$, $T>0$ and $\delta\in (0,T)$, the following holds
		\begin{align}
			\lim_{\delta\rightarrow 0} \sup_{t\in [0,T-\delta]}	\sup_{n\geq 1}	\left|\int_{0}^{\infty}  \left[\xi_{n}(t+\delta,u)-\xi_{n}(t,u)\right]\vartheta(u)\;du \right| =0.\label{equicontinuity-0}
		\end{align}
	\end{lemma}
	\begin{proof}
Consider $\vartheta \in  C_{c}^{1}([
		0,\infty))$. Fix $T>0$ and let $n\geq 1$, $\delta\in (0,T)$, and $t\in [0,T-\delta]$. Then, from \eqref{truncated weak formulation}, we have the following inequality 
		\begin{align}\label{EQN-1}
			&\left|\int_{0}^{\infty}(\xi_{n}(t+\delta,u)-\xi_n(t,u))\vartheta(u)\;du\right|\nonumber\\	\leq &{\|\vartheta^{\prime}\|_{L^{\infty}([0,\infty))}} \int_{t}^{t+\delta} \int_{0}^{\infty}|g_{n}(u)|\xi_{n}(s,u) \;duds+ \|\vartheta\|_{L^{\infty}([0,\infty))}\int_{t}^{t+\delta}\int_{0}^{\infty}a_{n}(u)\xi_{n}(s,u)\;duds\nonumber\\&\quad +\|\vartheta\|_{L^{\infty}([0,\infty))}\int_{t}^{t+\delta}\int_{0}^{\infty} \mu_{n}(u)\xi_{n}(s,u)\;duds\nonumber  \\ &\quad +\|\vartheta\|_{L^{\infty}([0,\infty))}\int_{t}^{t+\delta}\int_{0}^{\infty}\left(\int_{0}^{u_{1}}\beta_{n}(u|u_{1})\;du+1\right) \alpha_{n}(u_{1})\xi_{n}(s,u_{1})\;du_{1}ds \nonumber \\
			+	&{\frac{3}{2}\|\vartheta\|_{L^{\infty}([0,\infty))} \int_{t}^{t+\delta}\int_{0}^{\infty}\int_{0}^{\infty}\Upsilon_{n}(u,u_{1})\xi_{n}(s,u)\xi_{n}(s,u_{1}) \;dudu_{1}ds}.\nonumber
		\end{align}
	In the above calculations, we have used the fact that $|\vartheta(0)|\leq \sup_{u\in [0,\infty)}|\vartheta(u)|=\|\vartheta\|_{L^{\infty}([0,\infty))}$ for $\vartheta \in  C_{c}^{1}([
	0,\infty))$. Since $g$, $a$, $\mu$, and $\alpha$ belong to $Y_{\infty}$, from \eqref{trungmubetaF}, \eqref{truncoag}, \eqref{gnbound}, \eqref{particlebound}, \eqref{multiplicativecoagassum}, and \eqref{momentbound}, we obtain
		\begin{align}
			\left|\int_{0}^{\infty}(\xi_{n}(t+\delta,u)-\xi_n(t,u))\vartheta(u)\;du\right| \leq& C_{5}(T)\|\vartheta\|_{W^{1,\infty}([0,\infty))}\delta,
		\end{align}
	where 
	\begin{align*}
		C_{5}(T):=\left(1+\|g\|_{\infty}+\|a\|_{\infty}+\|\mu\|_{\infty}+(M+1)\|\alpha\|_{\infty}+\frac{3\Upsilon_{0}}{2}C_{1}e^{C_{3}T}\right)C_{1}e^{C_{3}T}.
	\end{align*}
Thus, \eqref{equicontinuity-0} holds for all $\vartheta \in C_c^1([0,\infty))$. By a density type argument, we extend its validity to all $\vartheta \in L^{\infty}([0,\infty))$; see~\cite[Proposition 4.7]{laurenccot2001weak}, \cite[Lemma 4.5]{Laurencot2001}, and \cite[Section 3.4]{barik2022mass} for reference. This completes the proof of \Cref{time equicontinuity}.  
\end{proof}  
Following~\Cref{time equicontinuity} and the definition in \cite[Definition 1.3.1]{Vrabie2003}, we deduce that  
\begin{align}\label{weakly equicontinuous}  
(\xi_n)_{n\geq 1} \text{ is weakly equicontinuous in } L^1([0,\infty)) \text{ at each } t \in [0,T].  
\end{align}  
Applying \eqref{weakly compact} and \eqref{weakly equicontinuous}, we invoke the Arzelà--Ascoli theorem \cite[Theorem 1.3.2]{Vrabie2003} to conclude that $(\xi_n)_{n\ge 1}$ is relatively compact in $C([0,T];w\text{-}L^1([0,\infty)))$ for every $T>0$. Then, we employ a diagonal argument to extract a sub-sequence $(\xi_n)_{n\geq 1}$ (not relabeled) and a function $\xi \in C([0,\infty); w\text{-}L^1([0,\infty)))$ satisfying
\begin{align}\label{convergence-3}  
\lim_{n\rightarrow \infty} \sup_{t\in [0,T]}\left| \int_{0}^{\infty}(\xi_n(t,u)-\xi(t,u))\vartheta(u)\;du\right|=0,
\end{align}  
for all $\vartheta \in L^{\infty}([0,\infty))$ and $T \in (0,\infty)$. Furthermore, using \eqref{momentbound}, Fatou’s lemma, and the non-negativity of $\left(\xi_n\right)_{n \ge 1}$ on $[0,T]$, we conclude that  
\begin{align}\label{limmomentbound}  
\|\xi(t)\| \leq C_{1}e^{C_{3}t},\quad  \xi(t) \in Y_{+} \quad \text{ for each } t \in [0,T],  
\end{align}  
and $\xi$ belongs to $L^{\infty}(0,T; Y_{+})$. Additionally, the convergence in \eqref{convergence-3} is further improved in the following proposition.  
	\begin{proposition}\label{improved convergence}
		Assume that the parameters $\Upsilon$, $\beta$, $\alpha$, $\mu$, $g$, $a$, and $\xi^{{\mathrm{in}}}$ satisfy conditions \eqref{multiplicativecoagassum}, \eqref{particlebound}, \eqref{lmassconservation}, \eqref{alphaassum}, \eqref{localbound},
		\eqref{mu}, \eqref{growthassum}, \eqref{birthrateassum}, and \eqref{initialassum}. For any $T>0$, the following statements are true
		\begin{enumerate}[label=(\alph*)]
		\item  $\xi_{n}\rightarrow \xi$ in  $C\left([0,T];w-L^1([0, \infty); u du)\right)$ if $\Upsilon$ satisfies \eqref{additivecoagassum}.
		\item $\xi_{n}\rightarrow \xi$ in  $C\left([0,T];w-L^1([0, \infty); r(u) du)\right)$ if $\Upsilon$ satisfies \eqref{subquadraticcoagassum}.
	\end{enumerate}
	\end{proposition}
\begin{proof}[\textbf{Proof of \Cref{improved convergence}(a)}]  
	Let $T>0$, $t \in [0, T]$, and $R > 1$. From \eqref{j0}, \eqref{convergence-3}, and Fatou's lemma, we obtain  
	\begin{align}  
		\int_{0}^{\infty} j_{0}(u)\xi(t,u) \,du \leq L(T), \quad \text{for } t \in [0,T]. \label{j1}  
	\end{align}  
	Using \eqref{j0} and \eqref{j1}, we deduce that  
	\begin{align}  
		&\left| \int_{0}^{\infty} u(\xi_n(t,u)-\xi(t,u))\vartheta(u)\,du \right| \nonumber\\  
		&\leq \left| \int_{0}^{R} u(\xi_n(t,u)-\xi(t,u))\vartheta(u)\,du \right|  
		+ \left| \int_{R}^{\infty} u (\xi_n(t,u)-\xi(t,u))\vartheta(u)\,du \right| \nonumber\\  
		&\leq \left| \int_{0}^{R} u(\xi_n(t,u)-\xi(t,u))\vartheta(u)\,du \right| \nonumber\\  
		&\quad + \|\vartheta\|_{L^{\infty}(0,\infty)}\sup_{u\geq R} \left\{\frac{u}{j_{0}(u)}\right\} \int_{R}^{\infty} j_{0}(u)(\xi_n(t,u)+\xi(t,u))\;du \nonumber\\  
		&\leq \left| \int_{0}^{R} u(\xi_n(t,u)-\xi(t,u))\vartheta(u)\,du \right|  
		+ 2\|\vartheta\|_{L^{\infty}([0,\infty))} L(T) \sup_{u\geq R} \left\{\frac{u}{j_{0}(u)}\right\},\label{massimproveconvo-1}  
	\end{align}  
	for every $\vartheta \in L^{\infty}([0,\infty))$ and $t\in [0,T]$. First, we take the supremum over $[0,T]$ on the both sides of \eqref{massimproveconvo-1}. Then, we pass to the limit as $n \to \infty$ using \eqref{convergence-3}, and let $R \to \infty$ with the help of \eqref{convexinfty}, arriving at  
	\begin{align}  
		\lim_{n\rightarrow \infty} \sup_{t\in [0,T]} \left| \int_{0}^{\infty} u(\xi_n(t,u)-\xi(t,u))\vartheta(u) \,du \right| = 0, \nonumber  
	\end{align}  
	as required.  
\end{proof}  
\begin{proof}[\textbf{Proof of \Cref{improved convergence}(b)}] For $T>0$, $t \in [0, T ]$, and $R > 1$, we have
\begin{align}
	&\left| \int_{0}^{\infty}r(u)(\xi_n(t,u)-\xi(t,u))\vartheta(u)\;du\right|\nonumber\\
	&\leq \left| \int_{0}^{R}r(u)(\xi_n(t,u)-\xi(t,u))\vartheta(u)\;du\right|+\left| \int_{R}^{\infty}r(u)(\xi_n(t,u)-\xi(t,u))\vartheta(u)\;du\right|\nonumber\\
	&\leq \left| \int_{0}^{R}r(u)(\xi_n(t,u)-\xi(t,u))\vartheta(u)\;du\right|+\|\vartheta\|_{L^{\infty}(0,\infty)}\sup_{u\geq R}\left\{\frac{r(u)}{u}\right\} \int_{R}^{\infty}u(\xi_n(t,u)+\xi(t,u))\;du,\nonumber
\end{align}
	for every $\vartheta \in L^{\infty}([0,\infty))$. Then, by \eqref{momentbound} and \eqref{limmomentbound}, we obtain
\begin{align}
	\left| \int_{0}^{\infty}r(u)(\xi_n(t,u)-\xi(t,u))\vartheta(u)\;du\right|
	&\leq \left| \int_{0}^{R}r(u)(\xi_n(t,u)-\xi(t,u))\vartheta(u)\;du\right| \nonumber\\
	&\quad +2C_{1}e^{C_{3}T}\|\vartheta\|_{L^{\infty}([0,\infty))}\sup_{u\geq R}\left\{\frac{r(u)}{u}\right\}. \nonumber
\end{align} 
Since $u\mapsto r(u)\chi_{[0,R]}$ belongs to $L^{\infty}([0,\infty))$ by \eqref{sublimboundcoagassum}, we deduce from \eqref{convergence-3} that
\begin{align}
	&\lim_{n\rightarrow \infty} \sup_{t\in [0,T]}\left| \int_{0}^{\infty}r(u)(\xi_n(t,u)-\xi(t,u))\vartheta(u)\right|\nonumber\\
	&\leq 2C_{1}e^{C_{3}T}\|\vartheta\|_{L^{\infty}(0,\infty)}\sup_{u\geq R}\left\{\frac{r(u)}{u}\right\}. \nonumber
\end{align}
Next, taking the limit as $R \to \infty$ in the above inequality and using \eqref{sublimboundcoagassum}, we obtain  
\begin{align}  
	\lim_{n\rightarrow \infty} \sup_{t\in [0,T]}\left| \int_{0}^{\infty} r(u)(\xi_n(t,u)-\xi(t,u))\vartheta(u) \,du \right|=0. \nonumber  
\end{align}  
This confirms that $\xi_n \to \xi$ in $C\left([0,T]; w\text{-}L^1([0, \infty); r(u)du)\right)$.  
\end{proof}
	\subsection{Proof of \Cref{main theorem} }
		At this stage, we are in a position to provide the proof of \Cref{main theorem}. To finalize the proof, it suffices to demonstrate that $\xi$ satisfies~\eqref{weak formulation}. Since the sequence $(\xi_{n})_{n\ge 1}$ satisfies the truncated weak formulation \eqref{truncated weak formulation}, we only have to pass to the limit in each of the corresponding terms. Clearly, \eqref{mollificationofinitialdata} implies that
	\begin{align}\label{initialdataconvergence}
		\lim_{n\rightarrow \infty}\int_{0}^{\infty} \xi_{n}^{{\mathrm{in}}}(u)\vartheta(u)\;du =\int_{0}^{\infty} \xi^{{\mathrm{in}}}(u)\vartheta(u)\;du,
	\end{align}
for any $\vartheta \in C^{1}_c\left([0,\infty)\right)$.
\begin{proof}[\textbf{Proof of \Cref{main theorem}(a)}]
	First, we consider $\vartheta \in C^{1}_c\left([0,\infty)\right)$, $T>0$, and $t\in [0,T]$. From \eqref{growthassum}, \eqref{gnbound}, and \eqref{trungmubetaF}, we have 
	\begin{align}\label{growthconvergencebound}
		\frac{|\vartheta^{\prime}(u)|g_{n}(u)}{1+u}\leq \|\vartheta^{\prime}\|_{L^{\infty}(0,\infty)}\left(\|g\|_{\infty}+1\right),\quad u\in (0,\infty), \;n\ge 1,
	\end{align}
and 
\begin{align}\label{growthpointwiseconvergence}
	\lim_{n\rightarrow \infty} \frac{g_{n}(u)}{1+u}= \frac{g(u)}{1+u},\quad u\in (0, \infty).
\end{align}
By \eqref{convergence-3} and \Cref{improved convergence}(a), we obtain that
\begin{align}\label{weightedweakconvergence}
		\xi_{n}\rightarrow \xi \quad {\mathrm{in}}\quad  C\left([0,T],w-Y\right)\quad \text{for all}\quad T>0.
\end{align}
Based on \eqref{growthconvergencebound}, \eqref{growthpointwiseconvergence}, \eqref{weightedweakconvergence}, and \eqref{momentbound}, it follows from \cite[Proposition 2.61]{fonseca2007modern} and Lebesgue's DCT that
\begin{align*}
	\lim_{n \to \infty} \int_{0}^{t} \int_{0}^{\infty} \frac{\vartheta^{\prime}(u) g_{n}(u)}{1+u} (1+u) \xi_{n}(s,u)\;duds = \int_{0}^{t} \int_{0}^{\infty} \frac{\vartheta^{\prime}(u) g(u)}{1+u} (1+u) \xi(s,u)\;duds.
\end{align*}
In other words,
\begin{align}\label{gconvergence}
	\lim_{n \to \infty} \int_{0}^{t} \int_{0}^{\infty} \vartheta^{\prime}(u) g_{n}(u) \xi_{n}(s,u)\;duds = \int_{0}^{t} \int_{0}^{\infty} \vartheta^{\prime}(u) g(u) \xi(s,u)\;duds.
\end{align}
Similarly, from \eqref{trungmubetaF}, \eqref{birthrateassum}, \eqref{mu}, \eqref{particlebound}, \eqref{alphaassum}, \eqref{weightedweakconvergence}, \eqref{momentbound}, \cite[Proposition 2.61]{fonseca2007modern}, and Lebesgue's DCT, we deduce the following convergence results
\begin{align}\label{betaconvergence}
\lim_{n\rightarrow \infty}\vartheta (0)\int_{0}^{t}\int_{0}^{\infty} a_{n}(u)\xi_{n}(s,u)\;duds= \vartheta (0)\int_{0}^{t}\int_{0}^{\infty} a(u)\xi(s,u)\;duds,	
\end{align}
\begin{align}\label{muconvergence}
	\lim_{n\rightarrow \infty}\int_{0}^{t}\int_{0}^{\infty} \vartheta (u)\mu_{n}(u)\xi_{n}(s,u)\;duds= \int_{0}^{t}\int_{0}^{\infty} \vartheta(u)\mu(u)\xi(s,u)\;duds,	
\end{align}
and 
\begin{align}\label{Fconvergence}
\lim_{n\rightarrow \infty}\int_{0}^{t}\int_{0}^{\infty}\psi(\vartheta)(u_{1}) \alpha_{n}(u_{1})\xi_{n}(s,u_{1})\;du_{1}ds=\int_{0}^{t}\int_{0}^{\infty}\psi(\vartheta)(u_{1}) \alpha(u_{1})\xi(s,u_{1})\;du_{1}ds.
\end{align}
Next, by applying \eqref{momentbound}, \eqref{convergence-3}, and \Cref{improved convergence}(a), we establish that  
\begin{align}\label{weakconvergence}  
	H_n \to H \quad {\mathrm{in}} \quad C\left([0,T], Y_{0,w} \times Y_{0,w} \right),  
\end{align}  
 where  
\begin{align}  
	H_n(t,u,u_{1}) := (1+u)(1+u_{1})\xi_n(t,u)\xi_n(t,u_{1}), \quad  
	H(t,u,u_{1}) := (1+u)(1+u_{1})\xi(t,u)\xi(t,u_{1}),\nonumber  
\end{align}  
for $(t,u,u_{1}) \in [0,T] \times [0,\infty)^2$, and $T>0$. The space $Y_{0,w}$ is given by $Y_{0,w} := w\text{-}L^1([0,\infty)) = L^1([0,\infty))$ endowed with the weak topology. Furthermore, utilizing \eqref{multiplicativecoagassum}, we observe that  
\begin{align}\label{convergencebound}  
	\frac{\left| \tilde{\vartheta}(u,u_{1}) \right| \Upsilon_n(u,u_{1})}{(1+u)(1+u_{1})} \leq 3\Upsilon_{0} \|\vartheta\|_{L^{\infty}([0,\infty))}, \quad (u,u_{1}) \in (0,\infty)^2, \quad n \geq 1.  
\end{align}  
Additionally, from \eqref{truncoag}, we deduce the pointwise limit  
\begin{align}\label{pointwiseconvergence}  
	\lim_{n \to \infty} \frac{\tilde{\vartheta}(u,u_{1}) \Upsilon_n(u,u_{1})}{(1+u)(1+u_{1})} = \frac{\tilde{\vartheta}(u,u_{1}) \Upsilon(u,u_{1})}{(1+u)(1+u_{1})}, \quad (u,u_{1}) \in (0,\infty)^2.  
\end{align}  
By using \eqref{weakconvergence}, \eqref{convergencebound}, \eqref{pointwiseconvergence}, \eqref{multiplicativecoagassum}, and \eqref{momentbound}, we apply \cite[Proposition 2.61]{fonseca2007modern} and Lebesgue's DCT to conclude that  
\begin{align}  
	\lim_{n\rightarrow \infty} \int_{0}^{t} \int_{0}^{\infty} \int_{0}^{\infty} \frac{\tilde{\vartheta}(u,u_{1}) \Upsilon_n(u,u_{1})}{(1+u)(1+u_{1})} H_n(s,u,u_{1}) \,du \,du_{1} \,ds \nonumber\\  
	= \int_{0}^{t} \int_{0}^{\infty} \int_{0}^{\infty} \frac{\tilde{\vartheta}(u,u_{1}) \Upsilon(u,u_{1})}{(1+u)(1+u_{1})} H(s,u,u_{1}) \,du \,du_{1} \,ds. \nonumber  
\end{align}  
Equivalently, this implies  
\begin{align}\label{coagconvergence}  
	\lim_{n\rightarrow \infty} \int_{0}^{t} \int_{0}^{\infty} \int_{0}^{\infty} \tilde{\vartheta}(u,u_{1}) \Upsilon_n(u,u_{1}) \xi_n(s,u)\xi_n(s,u_{1}) \,du \,du_{1} \,ds \nonumber\\  
	= \int_{0}^{t} \int_{0}^{\infty} \int_{0}^{\infty} \tilde{\vartheta}(u,u_{1}) \Upsilon(u,u_{1}) \xi(s,u)\xi(s,u_{1}) \,du \,du_{1} \,ds.  
\end{align}  
The identities \eqref{convergence-3}, \eqref{initialdataconvergence}, \eqref{gconvergence},  \eqref{betaconvergence},  \eqref{muconvergence},  \eqref{Fconvergence}, and  \eqref{coagconvergence}  ensure that $\xi$ satisfies \eqref{weak formulation} and complete the proof of  \Cref{main theorem}(a).
\end{proof}
\begin{corollary}\label{existenceofsumkernel}
	Under the assumptions of \Cref{uniqueness theorem}, at least one global weak solution $\xi$ to~\eqref{maineq.1}--\eqref{maineq.2} is obtained, satisfying \eqref{solsecondmomentbound}.
\end{corollary}
\begin{proof}
	The existence of a weak solution $\xi$ that satisfies \eqref{solsecondmomentbound} follows immediately from \Cref{main theorem}(a) and \Cref{second moment lemma-1}.
\end{proof}
\begin{proof}[\textbf{Proof of \Cref{main theorem}(b)}] Since $g$, $a$, $\mu$, and $\alpha$ belong to $Y_r$, we apply \eqref{subquadraticcoagassum}, \eqref{sublimboundcoagassum}, \eqref{trungmubetaF}, \eqref{truncoag}, \eqref{convergence-3}, \eqref{initialdataconvergence}, \Cref{improved convergence}(b), \cite[Proposition 2.61]{fonseca2007modern}, and Lebesgue's DCT to proceed similarly to the proof of \Cref{main theorem}(a) and establish that $\xi$ satisfies \eqref{weak formulation}. Hence, the proof is complete.
\end{proof}
%%%%%%%%%%%%%%%%%%%%%%%%%%%%%%%%%%%%%%%%%%
%%%%%%%%%%%%%%%%%%%%%%%%%%%%%%%%%%%%%%%%%%
%%%%%%%%%%%%%%%%%%%%%%%%%%%%%%%%%%%%%%%%%%
\section{Uniqueness}\label{uniqueness section}
\noindent In this section, we focus on proving the uniqueness of weak solutions to \eqref{maineq.1}--\eqref{maineq.2}, as stated in \Cref{uniqueness theorem}. The existence of weak solutions follows directly from the result in \Cref{existenceofsumkernel}, so it remains only to prove uniqueness. Establishing uniqueness in coagulation-fragmentation models typically requires careful estimates on the coagulation and fragmentation terms, ensuring that small perturbations in the initial data do not lead to drastically different evolution. Similar techniques have been employed in previous studies on Smoluchowski’s coagulation equation and its variations, see~\cite{Stewart1990, giri2013uniqueness}, where uniqueness is often obtained under suitable growth conditions on the fragmentation and coagulation kernels. 

Throughout this section, we assume that the parameters $\Upsilon$, $\beta$, $\mu$, $\alpha$, $g$, $a$, and $\xi^{{\mathrm{in}}}$ satisfy the conditions of \Cref{uniqueness theorem}. Furthermore, we use the term \emph{solution to \eqref{maineq.1}--\eqref{maineq.2}} specifically to refer to a \emph{solution, as defined in \Cref{definitionofsolution}, that satisfies \eqref{solsecondmomentbound}}. The key ingredient in proving uniqueness is the stability result, which we now establish as a fundamental step in our argument. The following stability result ensures that any two solutions with the same initial data must coincide, thereby confirming the uniqueness statement in \Cref{uniqueness theorem}.

%%%%%%%%%%%%%%%%
\begin{proposition} \label{continuousdependence result} 
	Let the hypotheses of \Cref{uniqueness theorem} be satisfied. Consider $\xi_{1}^{{\mathrm{in}}}$ and $\xi_{2}^{{\mathrm{in}}}$ as non-negative functions in $L^{1}\left([0, \infty); (1+u^{2})du\right)$. Let $T > 0$, and suppose $\xi_1$ and $\xi_2$ are weak solutions to \eqref{maineq.1}--\eqref{maineq.2} on $[0,T]$, corresponding to the initial data $\xi_{1}^{{\mathrm{in}}}$ and $\xi_{2}^{{\mathrm{in}}}$, respectively. Suppose further that
	\begin{equation}  
		\xi_i \in L^\infty\left(0,T;L^{1}\left([0, \infty); (1+u^{2})du\right)\right),\label{z014}  
	\end{equation}
for $i=1,2$. Then, a constant $L(T) > 0 $ is determined, which depends only on $T$, $\Upsilon_{0}$, $P_{\alpha}$, $ Q_{\alpha}$, $M$, $a_{1}$, $g_{0}$, $g_{1}$, $A$, $ \|\mu\|_{L^{\infty}([0,\infty))}$, and the $ L^\infty\left(0,T;L^{1}\left([0, \infty); (1+u^{2})du\right)\right)$-norms of $\xi_1$ and $\xi_2$, such that  
\begin{equation*}  
	\|\xi_{1}(t)-\xi_{2}(t)\|\leq L(T)	\|\xi_{1}^{{\mathrm{in}}}-\xi_{2}^{{\mathrm{in}}}\|, 
\end{equation*}
for every $t\in [0,T]$.  
\end{proposition}
%%%%%%%%%%%%%%%%
We will prove \Cref{continuousdependence result} in the same vein as \cite[Proposition 5.1]{Giri2025well}, proceeding in multiple steps, starting with the derivation of an integral inequality for a weighted $L^1$-norm of the difference $\xi_1 - \xi_2$ in \Cref{lem.diffineq}. For this purpose, let us introduce the following notation 
\begin{align*}  
	& e := \xi_1 - \xi_2, \quad e_0 := e(0,\cdot) = \xi_{1}^{{\mathrm{in}}} - \xi_{2}^{{\mathrm{in}}}, \quad \mathcal{F}_{d} := \mathcal{F}(\xi_{1}) - \mathcal{F}(\xi_{2}), \\  
	& \mathcal{C}_{d} := \mathcal{C}(\xi_{1}) - \mathcal{C}(\xi_{2}), \quad \text{and} \quad \nu(t) := \int_{0}^{\infty} \zeta(u) |e(t,u)| \,du,  
\end{align*}  
for $(t,u) \in [0,T] \times (0,\infty)$, where the weight function $\zeta$ is defined as  
\begin{equation*}  
	\zeta(u):= 1 + |u|, \quad u \in \mathbb{R}.  
\end{equation*}
Then the following lemma is an immediate implication of \eqref{z014}, \eqref{multiplicativecoagassum}, \eqref{alphaassum}, \eqref{particlebound}, and \eqref{lmassconservation}. 
%%%%%%%%%%%%%%%%
\begin{lemma}\label{lem.Swd}
	$e\in  L^\infty\left(0,T;L^{1}\left([0, \infty); (1+u^{2})du\right)\right)$ and ~$\mathcal{C}_{d}, \mathcal{F}_{d}\in  L^\infty\left(0,T;L^{1}\left([0, \infty); (1+u)du\right)\right)$ for every $T>0$.
\end{lemma}
%%%%%%%%%%%%%%%%

\begin{proof}
	For $T>0$, it follows directly from~\eqref{z014} that $e$ belongs to $L^\infty\left(0,T;L^{1}\left([0, \infty); (1+u^{2})du\right)\right)$. Given $t \in [0,T]$, we apply~\eqref{multiplicativecoagassum} along with Fubini’s theorem to derive the following inequality for each $i \in \{1,2\}$
	\begin{align}
		\int_0^\infty \zeta(u) |\mathcal{C}(\xi_i)(t,u)|\;du & \le \frac{1}{2} \int_0^\infty \int_0^\infty \zeta(u+u_{1}) \Upsilon(u,u_{1}) \xi_i(t,u) \xi_i(t,u_{1})\;du_{1}du\nonumber\\
		& \quad + \int_0^\infty \int_0^\infty \zeta(u) \Upsilon(u,u_{1}) \xi_i(t,u) \xi_i(t,u_{1})\;du_{1}du \nonumber\\
		& \le 2 \int_0^\infty \int_0^\infty \zeta(u) \Upsilon(u,u_{1}) \xi_i(t,u) \xi_i(t,u_{1})\;du_{1}du \nonumber\\
		& \le 2\Upsilon_{0} \int_0^\infty \int_0^\infty (1+u)^{2}(1+u_{1}) \xi_i(t,u) \xi_i(t,u_{1})\;du_{1}du \nonumber\\
		& \leq 4\Upsilon_{0}\left[M_{0}(\xi_i(t))+M_{2}(\xi_i(t))\right]\left[M_{0}(\xi_i(t))+M_{1}(\xi_i(t))\right]\label{Coagterm}. 
	\end{align} 
	Similarly, by using Fubini's theorem, \eqref{alphaassum}, \eqref{particlebound}, and \eqref{lmassconservation}, we obtain
	\begin{align}
		\int_0^\infty \zeta(u) |\mathcal{F}(\xi_i)(t,u)|\;du & \le   \int_0^\infty  \zeta(u)\alpha(u)\xi_i(t,u) \;du + \int_0^\infty \int_u^\infty \zeta(u) \alpha(u_{1}) \beta(u|u_{1})  \xi_i(t,u_{1})\;du_{1}du \nonumber\\
		&  \le   \int_0^\infty  (1+u)\left(P_{\alpha}u+Q_{\alpha}\right)\xi_i(t,u) \;du \nonumber\\
		&\quad + \int_0^\infty \int_{0}^{u_{1}} (1+u) \alpha(u_{1}) \beta(u|u_{1})  \xi_i(t,u_{1})\;dudu_{1} \nonumber\\
		& = P_{\alpha}M_{2}(\xi_{i}(t))+\left(P_{\alpha}+Q_{\alpha}\right)M_{1}(\xi_{i}(t))+Q_{\alpha}M_{0}(\xi_{i}(t))\nonumber\\
		&\quad +\int_0^\infty  \left(n(u_{1})+u_{1}\right) \left(P_{\alpha}u_{1}+Q_{\alpha}\right)  \xi_i(t,u_{1})\;du_{1}\nonumber\\
		& \le 2P_{\alpha}M_{2}(\xi_{i}(t))+\left[(M+1)P_{\alpha}+2Q_{\alpha}\right]M_{1}(\xi_{i}(t))\nonumber\\
		&\quad +\left[(M+1)Q_{\alpha}\right]M_{0}(\xi_{i}(t)),\label{Fragterm}
	\end{align}
for each $t\in [0,T]$ and $i \in \{1,2\}$. From~\eqref{z014}, we conclude that the right-hand sides of \eqref{Coagterm} and \eqref{Fragterm} belong to $L^\infty(0,T)$. Consequently, both $\mathcal{C}(\xi_{i})$ and $ \mathcal{F}(\xi_{i})$ lie in $ L^\infty\left(0,T;L^{1}\left([0, \infty); (1+u)du\right)\right)$ for $i \in \{1,2\}$. As a result, the differences $\mathcal{C}_{d} = \mathcal{C}(\xi_{1}) - \mathcal{C}(\xi_{2})$ and $ \mathcal{F}_{d} = \mathcal{F}(\xi_{1}) - \mathcal{F}(\xi_{2})$ also belong to the same space $ L^\infty\left(0,T;L^{1}\left([0, \infty); (1+u)du\right)\right)$.
\end{proof}
Next, we aim to derive the following inequality.
%%%%%%%%%%%%%%%%
\begin{lemma}\label{lem.diffineq}
Let the hypotheses of \Cref{uniqueness theorem} hold. For any $T > 0$ and $t \in [0,T]$, a positive constant $L_0$ is determined, depending solely on $a_{1}$, $g_{0}$, $g_{1}$, $A$, and $\|\mu\|_{L^{\infty}([0,\infty))}$, such that  
\begin{equation}  
	\nu(t) \le \nu(0) + L_0 \int_0^t \nu(s) \;ds + \int_0^t \int_0^\infty \zeta(u) \left[\mathcal{C}_{d}(s,u)+\mathcal{F}_{d}(s,u)\right]\, \mathrm{sign}(e(s,u)) \;duds. \label{z001}  
\end{equation}
\end{lemma}
%%%%%%%%%%%%%%%%
It is important to note that all terms in \eqref{z001} are well-defined as a consequence of \eqref{z014} and \Cref{lem.Swd}.

In a rigorous framework, \Cref{lem.diffineq} is derived by integrating over both volume and time variables after multiplying the equation governing $e$ (obtained from~\eqref{maineq.1} for $\xi_1$ and $\xi_2$) by $\zeta\,\mathrm{sign}(e)$. As $\xi_1$ and $\xi_2$ are not differentiable on $(0,T)\times (0,\infty)$, a direct application of \eqref{maineq.1} is not feasible. To overcome this limitation, we employ an appropriate regularization inspired by the theory developed by DiPerna and Lions; see~\cite{diperna1989ordinary} and \cite[Appendix~6.1 and~6.2]{perthame2006transport}.
 More precisely, we begin by extending $e$, $\mathcal{C}_d$, and $\mathcal{F}_d$ to $[0,T] \times \mathbb{R}$, and $a$, $\mu$, $g$, and $e_0$ to $\mathbb{R}$. We denote their extended versions as $\tilde{e}$, $\tilde{\mathcal{C}}_d$, $\tilde{\mathcal{F}}_d$, $\tilde{a}$, $\tilde{\mu}$, $\tilde{g}$, and $\tilde{e}_0$, respectively. These extensions are defined as follows
\begin{align*}
	\tilde{e}(t,u):=\begin{cases}
		e(t,u)\ &\text{ for }\ u\in (0,\infty),\\
		0\ &\text{otherwise},
	\end{cases}\quad 
\tilde{\mathcal{C}_{d}}(t,u):=\begin{cases}
	\mathcal{C}_{d}(t,u)\ &\text{ for }\ u\in (0,\infty),\\
	0\ &\text{otherwise},
	\end{cases} 
\end{align*}
\begin{align*}
	\tilde{\mathcal{F}_{d}}(t,u):=\begin{cases}
		\mathcal{F}_{d}(t,u)\ &\text{ for }\ u\in (0,\infty),\\
		0\ &\text{otherwise},
	\end{cases}\quad
	\tilde{a}(u):=\begin{cases}
		a(u)\ &\text{ for }\ u\in [0,\infty),\\
		0\ &\text{ otherwise },
	\end{cases}
\end{align*}
\begin{align*}
	\tilde{\mu}(u):=\begin{cases}
		\mu(u)\ &\text{ for }\ u\in [0,\infty),\\
		0\ &\text{otherwise}, 
	\end{cases}\quad
		\tilde{g}(u):=\begin{cases}
			g(u)\ &\text{for}\ u\in [0,\infty),\\
			g(0)\ &\text{otherwise}, 
	\end{cases}
\end{align*}
\begin{align*}
	\text{and}\quad
	\tilde{e_{0}}(u):=\begin{cases}
	e(0,u)\ &\text{ for }\ u\in (0,\infty),\\
	0\ &\text{otherwise},
	\end{cases}
\end{align*}
for each $t \in [0,T]$ and $u\in \mathbb{R}$. It follows directly from~\eqref{growthassum} that $ \tilde{g}$ is weakly differentiable on $\mathbb{R}$, and its weak derivative is given by 
\begin{align*}
\tilde{g}^{\prime}(u) =
	\begin{cases}
		g^{\prime}(u) & \text{ for } u \in (0, \infty),\\
			0 & \text{ for } u \in (-\infty, 0).
	\end{cases}
\end{align*}
Specifically, for every bounded interval  $I\subset\mathbb{R}$, $\tilde{g}(\cdot)$ is an element of $W^{1,\infty}(I)$.  As a result, $\tilde{g}(\cdot)$ is absolutely continuous on $I$, which leads to the identity 
\begin{align}\label{growthtildelipschitz}
	\tilde{g}(u)-\tilde{g}(u_{1})=\int_{u_{1}}^{u} \tilde{g}^{\prime}(x)dx, \quad (u,u_{1})\in\mathbb{R}^2.
\end{align}
Furthermore, from~\eqref{Lipschitz growth} and \eqref{growthassum}, it follows that 
\begin{equation}\label{growthtilde}
	|\tilde{g}^{\prime}(u)|\leq A \;\;\text{ a.e.},\;\text{and}\quad |\tilde{g}(u)|\leq \max\left\{g_{0}, g_{1}\right\}\left(1+|u|\right), \quad u\in \mathbb{R}.
\end{equation}
Next, we consider a family of mollifiers $\left(\rho_{\delta}\right)_{\delta>0}$ defined by  
\begin{align*}
\rho_{\delta}(u) := \frac{1}{\delta} \rho\left(\frac{u}{\delta}\right), \quad u \in \mathbb{R}, \quad \delta \in (0,1),
\end{align*}  
where $0\leq \rho \in C_{c}^{\infty}(\mathbb{R})$ with $\operatorname{supp}(\rho) \subset (-1,1)$ and $\|\rho\|_{L^{1}(\mathbb{R})} = 1$.  
We then define the regularized functions as follows  
\begin{align*}
\tilde{e}^{\delta} := \tilde{e} \star \rho_{\delta}, \quad  
\tilde{e_{0}}^{\delta} := \tilde{e_{0}} \star \rho_{\delta}, \quad  
\tilde{\mathcal{C}_{d}}^{\delta} := \tilde{\mathcal{C}_{d}} \star \rho_{\delta}, \quad  
\text{and} \quad \tilde{\mathcal{F}_{d}}^{\delta} := \tilde{\mathcal{F}_{d}} \star \rho_{\delta}.
\end{align*}

For later use, we recall a classical result on convolution; see~\cite[Theorem 2.29]{adams2003sobolev} and~\cite[Lemma 5.4]{Giri2025well}.
%%%%%%%%%%%%%%%%
\begin{lemma}\label{lem.convolution}
	Suppose $m\ge 0$ and $f \in L^1(\mathbb{R},(1+|u|^m)du)$. For each $\delta\in (0,1)$, the convolution $\rho_\delta\star f$ is a $C^\infty$-smooth function that remains in the weighted Lebesgue space $L^1(\mathbb{R},(1+|u|^m)du)$. Moreover, it satisfies the inequality
	\begin{equation*}
		\int_{\mathbb{R}} |(\rho_\delta\star f)(u)| (1+|u|^m)du \le 2^{m+1} \int_{\mathbb{R}} |f(u)| (1+|u|^m)du.
	\end{equation*}
	Additionally, we have the following convergence result
	\begin{equation*}
		\lim_{\delta\to 0} \int_{\mathbb{R}} |(\rho_\delta\star f - f)(u)| (1+|u|^m)\;du = 0.
	\end{equation*}
\end{lemma}
%%%%%%%%%%%%%%%%  

Next, we present an immediate implication of \Cref{lem.Swd} and \Cref{lem.convolution}.

%%%%%%%%%%%%%%%%
\begin{corollary}\label{cor.convdelta}
 The function $\tilde{e}^\delta$ belongs to $L^\infty(0,T; L^1(\mathbb{R}, (1+u^2)du))$, while $ \tilde{\mathcal{C}_{d}}^\delta$ and $ \tilde{\mathcal{F}_{d}}^\delta $ belong to $ L^\infty(0,T; L^1(\mathbb{R}, (1+|u|)du))$, for $0<\delta <1$.
\end{corollary}
%%%%%%%%%%%%%%%%
Next, using the definitions of $\tilde{e}$, $\tilde{e_{0}}$, $\tilde{\mathcal{C}_{d}}$, $\tilde{g}$, $\tilde{a}$, $\tilde{\mu}$, and $\tilde{\mathcal{F}_{d}}$, the following identity is directly obtained from \Cref{definitionofsolution}
\begin{align*}  
	\int_{\mathbb{R}} \tilde{e}(t,u)\vartheta(u)\;du &= \int_{\mathbb{R}} \tilde{e}_0(u)\vartheta(u)\;du + \int_0^t \int_{\mathbb{R}} \vartheta'(u) \tilde{g}(u) \tilde{e}(s,u)\;duds \\  
	& \quad + \vartheta(0) \int_0^t \int_{\mathbb{R}} \tilde{a}(u)\tilde{e}(s,u)\;duds - \int_0^t \int_{\mathbb{R}} \tilde{\mu}(u)\vartheta(u)\tilde{e}(s,u)\;duds \\  
	& \quad + \int_0^t \int_{\mathbb{R}} \tilde{\mathcal{C}_{d}}(s,u)\vartheta(u)\;duds + \int_0^t \int_{\mathbb{R}} \tilde{\mathcal{F}_{d}}(s,u)\vartheta(u)\;duds,  
\end{align*}  
for all $t \in [0,T]$ and $\vartheta \in C_{c}^{1}(\mathbb{R})$. To obtain a regularized equation, we substitute $\vartheta(u_{1}) = \rho_\delta(u-u_{1})$ for a fixed $u\in \mathbb{R}$ into the above identity. This yields that $\tilde{e}^\delta$ satisfies the following equation  
\begin{equation*}  
	\partial_t \tilde{e}^\delta + \partial_u \left( \rho_\delta \star (\tilde{g} \tilde{e}) \right) = \tilde{\mathcal{C}_{d}}^\delta + \tilde{\mathcal{F}_{d}}^\delta - (\tilde{\mu} \tilde{e}) \ast \rho_{\delta} + \rho_{\delta}(u) \int_{\mathbb{R}} \tilde{a}(u_{1})\tilde{e}(s,u_{1})\;du_{1}, \quad (t,u) \in (0,T) \times \mathbb{R}.  
\end{equation*}  
Alternatively, we can express it as  
\begin{equation}  
	\partial_t \tilde{e}^\delta + \partial_u \left( \tilde{g} \tilde{e}^\delta \right) = D^\delta + \tilde{\mathcal{C}_{d}}^\delta + \tilde{\mathcal{F}_{d}}^\delta - (\tilde{\mu} \tilde{e}) \ast \rho_{\delta} + \rho_{\delta}(u) \int_{\mathbb{R}} \tilde{a}(u_{1})\tilde{e}(t,u_{1})\;du_{1}, \quad (t,u) \in (0,T) \times \mathbb{R},  
	\label{moleq}  
\end{equation}  
where  
\begin{equation*}  
	D^\delta := \partial_u \left(\tilde{g} \tilde{e}^\delta - \rho_\delta \star (\tilde{g}\tilde{e})\right).  
\end{equation*}  
For each $\epsilon \in (0,1)$, we introduce the function  
\begin{equation*}  
	\Lambda_{\epsilon}(u) := \frac{u^{2}}{\sqrt{u^{2}+\epsilon}}, \quad u \in \mathbb{R}.  
\end{equation*}  
This function is continuously differentiable for all $\epsilon \in (0,1)$ and satisfies the following key estimates for all $u\in \mathbb{R}$,  
\begin{align}  
	& \big| \Lambda_{\epsilon}(u) - |u| \big| \leq \min\{\sqrt{\epsilon}, |u|\},\;\big| \Lambda_{\epsilon}'(u) - \operatorname{sign}(u) \big| \leq \frac{2\epsilon}{u^2+\epsilon} \leq 2,\; \text{and} \quad \big|\Lambda_{\epsilon}'(u)\big| \leq 2, \label{sigmapointwiseconvergence} \\  
	& \big| u \Lambda_{\epsilon}'(u) - \Lambda_\epsilon(u) \big| \leq \min\{\sqrt{\epsilon}, |u|\}. \label{sigmabound}  
\end{align}  
For $t\in (0,T)$ and $u\in\mathbb{R}$, we derive the following evolution equation by applying the chain rule and using~\eqref{moleq},  
\begin{align}  
	\partial_t \Lambda_{\epsilon}(\tilde{e}^{\delta}) + \partial_u \left( \tilde{g} \Lambda_{\epsilon}(\tilde{e}^{\delta}) \right) &= \Lambda_{\epsilon}'(\tilde{e}^{\delta}) \left[ D^\delta + \tilde{\mathcal{C}_{d}}^\delta + \tilde{\mathcal{F}_{d}}^\delta - (\tilde{\mu} \tilde{e}) \ast \rho_{\delta} + \rho_{\delta}(u) \int_{\mathbb{R}} \tilde{a}(u_{1})\tilde{e}(t,u_{1}) \;du_{1} \right] \nonumber \\  
	&\quad - \tilde{g}^{\prime} \left[ \tilde{e}^{\delta} \Lambda_{\epsilon}'(\tilde{e}^{\delta}) - \Lambda_{\epsilon}(\tilde{e}^{\delta}) \right]. \label{z004}  
\end{align}  

Before proceeding further, we verify in the following lemma that each term in~\eqref{z004} belongs to the appropriate function space.
%%%%%%%%%%%%%%%%
\begin{lemma}\label{convoestimate}
	For every $T>0$, there exists a positive constant $\mathcal{L}(T)$ that depends only on $T$, $\Upsilon_{0}$, $P_{\alpha}$, $Q_{\alpha}$, $M$, $a_{1}$, $g_{0}$, $g_{1}$, $A$, $\|\mu\|_{L^{\infty}(0,\infty)}$, $\rho$, $\xi_1$, and $\xi_2$ such that 
	\begin{subequations}\label{z005}
		\begin{align}
			\int_{\mathbb{R}} \zeta(u) \left| \partial_u\left(\tilde{g}(u) \Lambda_{\epsilon}(\tilde{e}^{\delta}(t,u))\right) \right| \;du & \le \frac{\mathcal{L}(T)}{\delta}, \label{z005a}\\
			\int_{\mathbb{R}} \zeta(u) \left| \Lambda_{\epsilon}'(\tilde{e}^{\delta}(t,u)) D^{\delta}(t,u) \right| \;du & \le \frac{\mathcal{L}(T)}{\delta}, \label{z005b}\\
			\int_{\mathbb{R}} \zeta(u) \left| \Lambda_{\epsilon}'(\tilde{e}^{\delta}(t,u)) \tilde{\mathcal{C}_{d}}^{\delta}(t,u) \right| \;du & \le \mathcal{L}(T), \label{z005c} \\
			\int_{\mathbb{R}} \zeta(u) \left| \Lambda_{\epsilon}'(\tilde{e}^{\delta}(t,u)) \tilde{\mathcal{F}_{d}}^{\delta}(t,u) \right| \;du & \le \mathcal{L}(T), \label{z005d} \\
			\int_{\mathbb{R}} \zeta(u) \left|\tilde{g}^{\prime}(u) \left[ \tilde{e}^{\delta} \Lambda_{\epsilon}'(\tilde{e}^{\delta}) - \Lambda_{\epsilon}(\tilde{e}^{\delta}) \right](t,u)  \right| \;du & \le \mathcal{L}(T), \label{z005e} \\
			\int_{\mathbb{R}} \zeta(u) \left| \Lambda_{\epsilon}'(\tilde{e}^{\delta}(t,u))\left((\tilde{\mu}\tilde{e})\ast \rho_{\delta}\right)(t,u) \right| \;du & \le \mathcal{L}(T), \label{z005f} \\
			\int_{\mathbb{R}} \zeta(u) \left| \Lambda_{\epsilon}'(\tilde{e}^{\delta}(t,u))  \rho_{\delta}(u) \int_{\mathbb{R}}  \tilde{a}(u_{1})\tilde{e}(t,u_{1})\;du_{1} \right| \;du & \le \mathcal{L}(T), \label{z005g} \\
		\text{and} \quad	\int_{\mathbb{R}} \zeta(u) \left| 	\partial_t	\Lambda_{\epsilon}(\tilde{e}^{\delta}(t,u)) \right| \;du & \le \frac{\mathcal{L}(T)}{\delta}, \label{z005h} 
		\end{align}
	for all $t\in [0,T]$.
	\end{subequations}
\end{lemma}	
%%%%%%%%%%%%%%%%

\begin{proof}
First, we proceed to prove~\eqref{z005a} and~\eqref{z005b}. Since
	\begin{equation*}
		\partial_{u}\left(\tilde{g}\Lambda_{\epsilon}(\tilde{e}^{\delta})\right) = \Lambda_{\epsilon}(\tilde{e}^{\delta}) \tilde{g}^{\prime}  + \tilde{g} \Lambda_{\epsilon}'(\tilde{e}^{\delta}) \big(\rho_\delta'\star\tilde{e}\big) %\label{z007}
	\end{equation*}
	and
	\begin{equation}
		D^\delta = \tilde{g} \big(\rho_\delta'\star\tilde{e}\big) + \tilde{e}^\delta \tilde{g}^{\prime} - \rho_\delta'\star\big(\tilde{g}\tilde{e}\big), \label{z008}
	\end{equation}
it is sufficient to examine the integrability of each term independently. Employing~\eqref{growthtilde} and \Cref{lem.convolution} with $m=1$, we deduce the following estimate
	\begin{align*}
		\int_{\mathbb{R}} \zeta(u) \left| \tilde{e}^{\delta}(t,u) \tilde{g}^{\prime}(u) \right| \;du 
		&\le A \int_{\mathbb{R}} \zeta(u) \left| \tilde{e}^{\delta}(t,u) \right| \;du \\ 
		& \le 4A \int_{\mathbb{R}} \zeta(u) \left| \tilde{e}(t,u) \right| \;du = 4A \int_0^\infty (1+u) |e(t,u)|\;du.
	\end{align*}
Similarly, using~\eqref{sigmapointwiseconvergence}, \eqref{growthtilde}, and \Cref{lem.convolution} with $m=1$, we infer that
\begin{align*}
	\int_{\mathbb{R}} \zeta(u) \left| \Lambda_{\epsilon}(\tilde{e}^{\delta}(t,u)) \tilde{g}^{\prime}(u) \right| \;du  
	& \le  8A \int_0^\infty (1+u) |e(t,u)|\;du.
\end{align*} 
	Next, applying~\eqref{growthtilde} and Fubini's theorem, we infer that
	\begin{align*}
		&\int_{\mathbb{R}} \zeta(u) \left| \tilde{g}(u)\big(\rho_\delta'\star\tilde{e}\big)(t,u) \right|\;du \\
		& \le \max\left\{g_{0}, g_{1}\right\} \int_{\mathbb{R}}  \int_{\mathbb{R}}  (1+|u|)^{2} |\rho_\delta'(u_{1})| |\tilde{e}(t, u-u_{1})| \;du_{1}du \\
		& \le {\delta}^{-1}\max\left\{g_{0}, g_{1}\right\} \int_{\mathbb{R}} \int_{\mathbb{R}} (1+|u+\delta u_{1}|)^2  |\rho'(u_{1})| |\tilde{e}(t,u)| \;du_{1}du \\
		& \le {\delta}^{-1}\max\left\{g_{0}, g_{1}\right\} \|\rho'\|_{L^1(\mathbb{R})} \int_{\mathbb{R}} (2+|u|)^2  |\tilde{e}(t, u)| \;du \\
		& \le 8{\delta}^{-1}\max\left\{g_{0}, g_{1}\right\} \|\rho'\|_{L^1(\mathbb{R})} \int_0^\infty (1+u^2)  |e(t, u)| \;du.
	\end{align*}
Similarly, applying~\eqref{sigmapointwiseconvergence}, \eqref{growthtilde}, and Fubini's theorem, we obtain
\begin{align*}
	\int_{\mathbb{R}} \zeta(u) \left| \tilde{g}(u) \Lambda_{\epsilon}'(\tilde{e}^{\delta}(t,u)) \big(\rho_\delta'\star\tilde{e}\big)(t,u) \right|\;du  \le  16{\delta}^{-1}\max\left\{g_{0}, g_{1}\right\} \|\rho'\|_{L^1(\mathbb{R})} \int_0^\infty (1+u^2)  |e(t, u)| \;du.
\end{align*}
Finally, consider the last term on the right-hand side of \eqref{z008}. Using~\eqref{growthtilde}, we obtain
	\begin{align*}
		&\int_{\mathbb{R}} \zeta(u) \left| \rho_\delta'\star\big(\tilde{g}\tilde{e}\big)(t, u) \right|\;du \\
		& \le \max\left\{g_{0}, g_{1}\right\} \int_{\mathbb{R}}  \int_{\mathbb{R}} \left(1+|u-u_{1}|\right) (1+|u|) |\rho_\delta'(u_{1})| |\tilde{e}(t,u-u_{1})| \;du_{1}du \\
		& \le  {\delta}^{-1}\max\left\{g_{0}, g_{1}\right\} \int_{\mathbb{R}} \int_{\mathbb{R}} (1+|u+\delta u_{1}|) \left(1+|u|\right)  |\rho'(u_{1})| |\tilde{e}(t,u)| \;du_{1}du \\
		& \le  8{\delta}^{-1}\max\left\{g_{0}, g_{1}\right\} \|\rho'\|_{L^1(\mathbb{R})} \int_0^\infty (1+u^2)  |e(t,u)| \;du.
	\end{align*}
By combining the above estimates with~\eqref{z014}, we establish~\eqref{z005a} and~\eqref{z005b}. The bound in~\eqref{z005c} and \eqref{z005d} directly follow from \eqref{sigmapointwiseconvergence}, \Cref{lem.Swd}, \Cref{lem.convolution}, and \eqref{z014}, whereas~\eqref{z005e} is an immediate result of~\eqref{z014}, \eqref{growthtilde}, \eqref{sigmabound}, and \Cref{lem.convolution}.  
	Next, we proceed to prove \eqref{z005f} and \eqref{z005g}. Since $\mu \in L^{\infty}([0,\infty))$, by using \eqref{sigmapointwiseconvergence} and Fubini's theorem, we obtain
	\begin{align}
		&\int_{\mathbb{R}} \zeta(u) \left| \Lambda_{\epsilon}'(\tilde{e}^{\delta}(t,u))\left((\tilde{\mu}\tilde{e})\ast \rho_{\delta}\right)(t,u) \right| \;du \nonumber \\
		&\leq 2 	\int_{\mathbb{R}} (1+|u|) \int_{\mathbb{R}} \left|\tilde{\mu}(u-u_{1})\tilde{e}(t,u-u_{1})\right|\rho_{\delta}(u_{1})  \;du_{1}du\nonumber \\
		&\leq 2 	\int_{\mathbb{R}} (1+|u_{1}+u|) \int_{\mathbb{R}} \left|\tilde{\mu}(u)\tilde{e}(t,u)\right|\rho_{\delta}(u_{1})  \;dudu_{1}\nonumber\\
		&\leq 2 \|\mu \|_{L^{\infty}(0,\infty)} \left(	\int_{\mathbb{R}} (1+|u_{1}|) \rho_{\delta}(u_{1}) \;du_{1} \right) \left(\int_{0}^{\infty} (1+|u|)|e(t,u)| \;du \right)\nonumber\\
		&\leq 4 \|\mu \|_{L^{\infty}(0,\infty)}  \int_{\mathbb{R}} \zeta(u) \left| \tilde{e}(t,u) \right|\;du.\label{mudeltabound} 
	\end{align}
From \eqref{sigmapointwiseconvergence} and \eqref{birthrateassum}, we have
	\begin{align}
	&\int_{\mathbb{R}} \zeta(u) \left| \Lambda_{\epsilon}'(\tilde{e}^{\delta}(t,u))  \rho_{\delta}(u) \int_{\mathbb{R}}  \tilde{a}(u_{1})\tilde{e}(t,u_{1})\;du_{1} \right| \;du\nonumber\\
	&\leq 2a_{1} \int_{\mathbb{R}} (1+|u|)   \rho_{\delta}(u) \int_{0}^{\infty}  (1+u_{1}) \left| \tilde{e}(t,u_{1}) \right|\;du_{1} du \nonumber\\
	&\leq 4 a_{1} \int_{\mathbb{R}} \zeta(u) \left| \tilde{e}(t,u) \right|\;du. \label{delatbetabound}
	\end{align}
	By utilizing the estimates from \eqref{mudeltabound} and \eqref{delatbetabound}, along with~\eqref{z014}, we derive \eqref{z005f} and \eqref{z005g}. Lastly, applying~\eqref{z004} together with \eqref{z005a}–\eqref{z005g}, we establish~\eqref{z005h}.
\end{proof}	

\begin{proof}[Proof of \Cref{lem.diffineq}]
	We infer from \eqref{z004} that
	\begin{align*}
		& \frac{d}{dt} \int_{\mathbb{R}} \zeta(u) \Lambda_{\epsilon}(\tilde{e}^{\delta})\;du + \int_{\mathbb{R}} \zeta(u)\partial_{u} \left( \tilde{g} \Lambda_{\epsilon}(\tilde{e}^{\delta}) \right) \;du\\
		& \quad = \int_{\mathbb{R}} \zeta(u) \left(\Lambda_{\epsilon}'(\tilde{e}^{\delta}) D^{\delta} + \Lambda_{\epsilon}'(\tilde{e}^{\delta}) \tilde{\mathcal{C}_{d}}^{\delta} + \tilde{g}^{\prime} \left[ \Lambda_{\epsilon}(\tilde{e}^{\delta}) - \tilde{e}^{\delta}\Lambda_{\epsilon}'(\tilde{e}^{\delta}) \right] \right)\;du\nonumber\\
		&\qquad +  \int_{\mathbb{R}} \zeta(u)\Lambda_{\epsilon}'(\tilde{e}^{\delta}) \left(\tilde{\mathcal{F}_{d}}^\delta - (\tilde{\mu}\tilde{e})\ast \rho_{\delta}+  \rho_{\delta}(u) \int_{\mathbb{R}}  \tilde{a}(u_{1})\tilde{e}(t,u_{1})\;du_{1}\right) du.
	\end{align*}
Note that all the integrals in the above equation are absolutely convergent according to \Cref{convoestimate}. Next, by performing integration by parts to the second term on the left-hand side of the preceding equation, we obtain
	\begin{align*}
		&\frac{d}{dt} \int_{\mathbb{R}} \zeta(u) \Lambda_{\epsilon}(\tilde{e}^{\delta})\;du - \int_{\mathbb{R}} \zeta'(u) \tilde{g} \Lambda_\epsilon(\tilde{e}^\delta) \;du\\
		 &\quad = \int_{\mathbb{R}} \zeta(u) \left(\Lambda_{\epsilon}'(\tilde{e}^{\delta}) D^{\delta} + \Lambda_{\epsilon}'(\tilde{e}^{\delta}) \tilde{\mathcal{C}_{d}}^{\delta} + \tilde{g}^{\prime} \left[ \Lambda_{\epsilon}(\tilde{e}^{\delta}) - \tilde{e}^{\delta}\Lambda_{\epsilon}'(\tilde{e}^{\delta}) \right] \right)\;du\nonumber\\
		&\qquad +  \int_{\mathbb{R}} \zeta(u)\Lambda_{\epsilon}'(\tilde{e}^{\delta}) \left(\tilde{\mathcal{F}_{d}}^\delta - (\tilde{\mu}\tilde{e})\ast \rho_{\delta}+  \rho_{\delta}(u) \int_{\mathbb{R}}  \tilde{a}(u_{1})\tilde{e}(t,u_{1})\;du_{1}\right) du.
	\end{align*}
Owing to~\eqref{growthtilde}, \eqref{sigmapointwiseconvergence}, and \Cref{cor.convdelta}, the boundary terms in the previous equation disappear. Next, for $t\in [0,T]$, integrating over $(0,t)$ yields
	\begin{align}
			\int_{\mathbb{R}} \zeta(u) \Lambda_{\epsilon}(\tilde{e}^{\delta}(t,u))\;du & =  \int_{\mathbb{R}} \zeta(u) \Lambda_{\epsilon}(\tilde{e}^{\delta}(0,u))\;du + \int_{0}^{t} \int_{\mathbb{R}} \zeta'(u) \tilde{g} \Lambda_\epsilon(\tilde{e}^\delta) \;duds \nonumber\\
			& \quad + \int_{0}^{t}\int_{\mathbb{R}} \zeta(u) \left(\Lambda_{\epsilon}'(\tilde{e}^{\delta}) D^{\delta} + \Lambda_{\epsilon}'(\tilde{e}^{\delta}) \tilde{\mathcal{C}_{d}}^{\delta} + \tilde{g}^{\prime} \left[ \Lambda_{\epsilon}(\tilde{e}^{\delta}) - \tilde{e}^{\delta}\Lambda_{\epsilon}'(\tilde{e}^{\delta}) \right] \right)\;du ds\nonumber\\
			&\quad +  \int_{0}^{t}\int_{\mathbb{R}} \zeta(u)\Lambda_{\epsilon}'(\tilde{e}^{\delta}) \left(\tilde{\mathcal{F}_{d}}^\delta - (\tilde{\mu}\tilde{e})\ast \rho_{\delta}+  \rho_{\delta}(u) \int_{\mathbb{R}}  \tilde{a}(u_{1})\tilde{e}(s,u_{1})\;du_{1}\right)\;duds.\label{epdelmain-1}
	\end{align}
	At this stage, we observe that~\eqref{growthtilde} and~\eqref{sigmapointwiseconvergence} imply
	\begin{align*}
		I_1 & := \left| \int_{\mathbb{R}} \zeta^{\prime}(u) \tilde{g}(u) \Lambda_\epsilon(\tilde{e}^\delta(t,u)) \;du \right| \le 2\max\left\{g_{0}, g_{1}\right\} \int_{\mathbb{R}} \zeta(u)\left| \tilde{e}^\delta(t,u) \right| \;du. 
	\end{align*}
	Since
	\begin{align}
		\int_{\mathbb{R}} \zeta(u) \left| \tilde{e}^\delta(t,u) \right| \;du & \le \int_{\mathbb{R}} \int_{\mathbb{R}} \zeta(u) \rho_\delta(u_{1}) \left| \tilde{e}(t,u-u_{1}) \right| \;du_{1}du \nonumber\\
		& = \int_{\mathbb{R}} \int_{\mathbb{R}} \zeta(u+\delta u_{1}) \rho(u_{1}) \left| \tilde{e}(t,u) \right| \;du_{1}du \nonumber \\
		& \le \int_{\mathbb{R}} \left[\zeta(u)+\delta\right] \left| \tilde{e}(t,u) \right|\;du, \label{z009}
	\end{align}
we deduce that
	\begin{equation}
		I_1 \le 2\max\left\{g_{0}, g_{1}\right\} \int_{\mathbb{R}}  \left[\zeta(u)+\delta\right]  \left| \tilde{e}(t,u) \right| \;du. \label{z010}
	\end{equation}
	Next, by~\eqref{z008}, we conclude from~\eqref{growthtilde} and~\eqref{sigmapointwiseconvergence} that
	\begin{align}
		I_2 & := \left| \int_{\mathbb{R}} \zeta(u) \Lambda_{\epsilon}'(\tilde{e}^{\delta}(t,u)) D^{\delta}(t,u) \;du \right| \nonumber \\& 
		\le 2 \int_{\mathbb{R}} \zeta(u) \left| \tilde{g}(u) \big(\rho_\delta'\star\tilde{e}\big)(t,u) - \rho_\delta'\star\big(\tilde{g}\tilde{e}\big)(t,u) \right| \;du + 2A \int_{\mathbb{R}} \zeta(u) \left| \tilde{e}^{\delta}(t,u) \right| \;du.\label{TempI2}
	\end{align}
	At this moment, due to~\eqref{growthtildelipschitz} and \eqref{growthtilde}, we obtain
	\begin{align}
		& \int_{\mathbb{R}} \zeta(u) \left| \tilde{g}(u) \big(\rho_\delta'\star\tilde{e}\big)(t,u) - \rho_\delta'\star\big(\tilde{g}\tilde{e}\big)(t,u) \right| \;du \nonumber\\
		& \quad \le \int_{\mathbb{R}} \zeta(u) \left| \int_{\mathbb{R}} \left[  \tilde{g}(u) \rho_\delta'(u_{1}) \tilde{e}(t,u-u_{1}) - \rho_\delta'(u_{1}) \tilde{g}(u-u_{1}) \tilde{e}(t,u-u_{1}) \right]\;du_{1} \right| \;du \nonumber\\
		& \quad \le \int_{\mathbb{R}} \int_{\mathbb{R}} \zeta(u) \left|  \tilde{g}(u) - \tilde{g}(u-u_{1}) \right| |\rho_\delta'(u_{1})| \left| \tilde{e}(t,u-u_{1}) \right| \;du_{1}du \nonumber\\
		& \quad \le A \int_{\mathbb{R}} \int_{\mathbb{R}} \zeta(u) |u_{1} \rho_\delta'(u_{1})| \left| \tilde{e}(t,u-u_{1}) \right| \;du_{1}du\nonumber\\
		& \quad = A \int_{\mathbb{R}} \int_{\mathbb{R}} \zeta(u+\delta u_{1}) |u_{1} \rho'(u_{1})| \left| \tilde{e}(t,u) \right| \;du_{1}du \nonumber\\
		& \quad \le A I_{\rho} \int_{\mathbb{R}}  \left[\zeta(u)+\delta\right]\left| \tilde{e}(t,u) \right| \;du,\label{TempAIRho}
	\end{align}
	with 
	\begin{equation*}
	I_{\rho} := \int_{\mathbb{R}} |u\rho'(u)| \;du.
	\end{equation*}
	By incorporating the two estimates from~\eqref{TempI2} and \eqref{TempAIRho}~with~\eqref{z009}, we arrive at 
	\begin{equation}
		I_2 \le 2A(1+I_{\rho}) \int_{\mathbb{R}} \left[\zeta(u)+\delta\right] \left| \tilde{e}(t,u) \right| \;du. \label{z011}
	\end{equation}
Next, applying~\eqref{growthtilde} and~\eqref{z009} once again, together with~\eqref{sigmabound}, we derive
	\begin{align}
		I_3 & := \left| \int_{\mathbb{R}} \zeta(u)  \tilde{g}^{\prime}(u) \left[ \Lambda_{\epsilon}(\tilde{e}^{\delta}(t,u)) - \tilde{e}^{\delta}(t,u) \Lambda_{\epsilon}'(\tilde{e}^{\delta}(t,u)) \right]\;du \right| \nonumber \\
		& \le A \int_{\mathbb{R}} \zeta(u) \left| \tilde{e}^{\delta}(t,u) \right| \;du \le A  \int_{\mathbb{R}} \left[\zeta(u)+\delta\right] \left| \tilde{e}(t,u) \right| \;du. \label{z012}
	\end{align}
	By combining the estimates~\eqref{mudeltabound}, \eqref{delatbetabound}, \eqref{z010}, \eqref{z011}, \eqref{z012}, and applying~\eqref{z009} once more, we deduce from~\eqref{epdelmain-1} that a constant $L_{0}>0$ exists, depending on $a_{1}$, $g_{0}$, $g_{1}$, $A$, and $\|\mu \|_{L^{\infty}(0,\infty)}$, such that
	\begin{equation}
		\begin{split}
			\int_{\mathbb{R}} \zeta(u) \Lambda_\epsilon(\tilde{e}^\delta(t,u))\;du & \le \int_{\mathbb{R}} \left[\zeta(u)+\delta\right] \left| \tilde{e}_{0} \right| \;du + L_{0} \int_0^t \int_{\mathbb{R}}  \left[\zeta(u)+\delta\right] \left| \tilde{e}(s,u) \right| \;duds\\
			&\quad + \int_{0}^{t} \int_{\mathbb{R}} \zeta(u)\left[ \tilde{\mathcal{C}_{d}}^\delta(s,u)+ \tilde{\mathcal{F}_{d}}^\delta(s,u)\right] \Lambda_{\epsilon}^{\prime}(\tilde{e}^\delta(s,u)) \;duds.
		\end{split} \label{epdelmain-2}
	\end{equation}
To take the limit as $\delta \rightarrow 0$ in~\eqref{epdelmain-2}, we analyze each term individually. Applying~\eqref{sigmapointwiseconvergence}, we first examine the following term
	\begin{equation*}
		\left|\int_{\mathbb{R}} \zeta(u) \left[	\Lambda_{\epsilon}(\tilde{e}^{\delta}(t,u)) - \Lambda_{\epsilon}(\tilde{e}(t,u)) \right] \;du \right| \leq 2 \int_{\mathbb{R}} \zeta(u) \left| \tilde{e}^{\delta}(t,u) - \tilde{e}(t,u) \right| \;du. 
	\end{equation*}
Using~\Cref{lem.convolution} for $m=1$, we can pass to the limit as $\delta\rightarrow 0$ in the above inequality, yielding
	\begin{equation}
		\lim_{\delta\to 0} \int_{\mathbb{R}} \zeta(u) \Lambda_{\epsilon}(\tilde{e}^{\delta}(t,u)) \;du = \int_{\mathbb{R}} \zeta(u)  \Lambda_{\epsilon}(\tilde{e}(t,u)) \;du = \int_{0}^{\infty} \zeta(u)  \Lambda_{\epsilon}(e(t,u)) \;du. \label{dellim-3}
	\end{equation}
	Similarly,
	\begin{align}
		&\left| \int_{0}^{t} \int_{\mathbb{R}} \zeta(u) \left[ \Lambda_{\epsilon}'(\tilde{e}^{\delta}(s,u)) \tilde{\mathcal{C}_{d}}^{\delta}(s,u) - \Lambda_{\epsilon}'(\tilde{e}(s,u)) \tilde{\mathcal{C}_{d}}(s,u) \right] \;duds \right|\nonumber\\
		&\leq \int_0^t \int_{\mathbb{R}} \zeta(u) \left| \Lambda_{\epsilon}'(\tilde{e}^{\delta}(s,u)) \right| \left| \tilde{\mathcal{C}_{d}}^{\delta}(s,u) - \tilde{\mathcal{C}_{d}}(s,u)\right| \;duds \nonumber\\
		& \quad + \int_{0}^{t} \int_{\mathbb{R}} \zeta(u) \left| \tilde{\mathcal{C}_{d}}(s,u) \right| \left| \Lambda_{\epsilon}'(\tilde{e}^{\delta}(s,u)) - \Lambda_{\epsilon}'(\tilde{e}(s,u)) \right| \;duds.\label{coagepsilon-1}
	\end{align}
	First, utilizing~\eqref{sigmapointwiseconvergence}, we obtain  
	\begin{equation*}  
		\int_0^t \int_{\mathbb{R}} \zeta(u) \left| \Lambda_{\epsilon}'(\tilde{e}^{\delta}(s,u)) \right| \left| \tilde{\mathcal{C}_{d}}^{\delta}(s,u) - \tilde{\mathcal{C}_{d}}(s,u)\right| \;duds  
		\leq 2 \int_{0}^{t} \int_{\mathbb{R}} \zeta(u) \left| \tilde{\mathcal{C}_{d}}^{\delta}(s,u) - \tilde{\mathcal{C}_{d}}(s,u)\right| \;duds.  
	\end{equation*}  
By applying~\Cref{lem.convolution} along with Lebesgue's DCT and~\Cref{lem.Swd}, we conclude that  
	\begin{equation}  
		\lim_{\delta\to 0} \int_{0}^{t} \int_{\mathbb{R}} \zeta(u) \left| \Lambda_{\epsilon}'(\tilde{e}^{\delta}(s,u)) \right| \left| \tilde{\mathcal{C}_{d}}^{\delta}(s,u) - \tilde{\mathcal{C}_{d}}(s,u)\right| \;duds = 0.  \label{coagepsilon-2}  
	\end{equation}  
Moreover, since $\zeta \tilde{\mathcal{C}_{d}}$ belongs to $L^1((0,T)\times\mathbb{R})$ due to~\Cref{lem.Swd}, and considering the Lipschitz continuity of $\Lambda_\epsilon'$, \Cref{lem.convolution}, and~\eqref{sigmapointwiseconvergence}, we can find a sub-sequence (not relabeled),  
	\begin{equation*}  
		\lim_{\delta\to 0} \left| \Lambda_{\epsilon}'(\tilde{e}^{\delta}(s,u)) - \Lambda_{\epsilon}'(\tilde{e}(s,u)) \right| = 0 \quad \text{for almost every } (s,u) \in (0,t)\times\mathbb{R},  
	\end{equation*}  
	with the uniform bound  
	\begin{equation*}  
		\left| \Lambda_{\epsilon}'(\tilde{e}^{\delta}(s,u)) - \Lambda_{\epsilon}'(\tilde{e}(s,u)) \right| \leq 4 \quad \text{for almost every } (s,u) \in (0,t)\times\mathbb{R}.  
	\end{equation*}  
At this point, we can apply Lebesgue’s DCT, leading to
\begin{equation}
	\lim_{\delta\to 0} \int_{0}^{t} \int_{\mathbb{R}} \zeta(u) \left| \tilde{\mathcal{C}_{d}}(s,u) \right| \left| \Lambda_{\epsilon}'(\tilde{e}^{\delta}(s,u)) - \Lambda_{\epsilon}'(\tilde{e}(s,u)) \right| \;duds = 0. \label{coagepsilon-3}
\end{equation}
By combining~\eqref{coagepsilon-1}, \eqref{coagepsilon-2}, and~\eqref{coagepsilon-3}, we obtain
\begin{align}
	\lim_{\delta\to 0} \int_{0}^{t} \int_{\mathbb{R}} \zeta(u) \Lambda_{\epsilon}'(\tilde{e}^{\delta}(s,u)) \tilde{\mathcal{C}_{d}}^{\delta}(s,u) \;duds &= \int_{0}^{t} \int_{\mathbb{R}} \zeta(u) \Lambda_{\epsilon}'(\tilde{e}(s,u)) \tilde{\mathcal{C}_{d}}(s,u) \;duds \nonumber \\
	&= \int_{0}^{t} \int_{0}^{\infty} \zeta(u) \Lambda_{\epsilon}'(e(s,u)) \mathcal{C}_{d}(s,u) \;duds. \label{coagepsilon-4}
\end{align}
Similarly, by using \Cref{lem.Swd}, \Cref{lem.convolution}, the Lipschitz continuity of $\Lambda_\epsilon'$, \eqref{sigmapointwiseconvergence}, and Lebesgue's DCT, we can show that
 \begin{align}
 	\lim_{\delta\to 0} \int_{0}^{t} \int_{\mathbb{R}} \zeta(u) \Lambda_{\epsilon}'(\tilde{e}^{\delta}(s,u)) \tilde{\mathcal{F}_{d}}^{\delta}(s,u) \;duds & = \int_{0}^{t} \int_{\mathbb{R}} \zeta(u) \Lambda_{\epsilon}'(\tilde{e}(s,u)) \tilde{\mathcal{F}_{d}}(s,u) \;duds \nonumber \\
 	& = \int_{0}^{t} \int_{0}^{\infty} \zeta(u) \Lambda_{\epsilon}'(e(s,u)) \mathcal{F}_{d}(s,u) \;duds. \label{coagepsilon-5}
 \end{align}
Using~\eqref{dellim-3}, \eqref{coagepsilon-4}, and~\eqref{coagepsilon-5}, we can take the limit as \(\delta \to 0\) in~\eqref{epdelmain-2}, yielding  
	\begin{equation}
		\begin{split}
			\int_{0}^{\infty} \zeta(u) \Lambda_\epsilon(e(t,u)) \;du & \le \int_{0}^\infty \zeta(u) \left| e_{0}(u) \right| \;du + \int_{0}^{t} \int_{0}^\infty \zeta(u) \left[\mathcal{C}_{d}(s,u)+\mathcal{F}_{d}(s,u)\right] \Lambda_{\epsilon}^{\prime}(e(s,u)) \;duds \\
			& \quad + L_0 \int_0^t \int_{0}^\infty \zeta(u) \left| e(s,u) \right| \;duds,
		\end{split} \label{epdelmain-3}
	\end{equation}
for every $t\in [0,T]$. Now, we proceed to pass the limit as $\epsilon \to 0$ in~\eqref{epdelmain-3}. Since $e$ belongs to $L^\infty\left(0,T;L^{1}\left([0, \infty); (1+u^{2})du\right)\right)$, by \Cref{lem.Swd}, we can apply~\eqref{sigmapointwiseconvergence} along with Lebesgue's DCT to establish the following limits  
	\begin{align*}
		\lim_{\epsilon\to 0} \int_0^\infty \zeta(u) \Lambda_\epsilon(e(t,u)) \;du & = \int_{0}^{\infty} \zeta(u) |e(t,u)| \;du, \\
		\lim_{\epsilon\to 0} \int_{0}^{\infty} \zeta(u) \Lambda_\epsilon(e_{0}(u)) \;du & = \int_{0}^{\infty} \zeta(u) |e_{0}(u)| \;du, \\
		\lim_{\epsilon\to 0} \int_{0}^{t} \int_{0}^{\infty} \zeta(u) \Lambda_\epsilon(e(s,u)) \;duds & = \int_{0}^{t} \int_{0}^{\infty} \zeta(u) |e(s,u)| \;duds. \\
	\end{align*}
	Furthermore, we recall the property 
	\begin{equation*}
		\lim_{\epsilon\to 0} \zeta(u) \left[\mathcal{C}_{d}(s,u)+ \mathcal{F}_{d}(s,u)\right] \Lambda_\epsilon'(e(s,u)) = \zeta(u) \left[\mathcal{C}_{d}(s,u)+ \mathcal{F}_{d}(s,u)\right] \,\mathrm{sign}(e(s,u)),
	\end{equation*}
	which holds for almost every $(s,u) \in (0,t) \times (0,\infty)$, with the bound
	\begin{equation*}
		\left| \zeta(u) \left[\mathcal{C}_{d}(s,u)+ \mathcal{F}_{d}(s,u)\right] \Lambda_\epsilon'(e(s,u)) \right| \le 2 \zeta(u) \left[|\mathcal{C}_{d}(s,u)|+|\mathcal{F}_{d}(s,u)|\right],
	\end{equation*}
following from~\eqref{sigmapointwiseconvergence}.  Therefore, applying \Cref{lem.Swd} and Lebesgue's DCT, we obtain the convergence
	\begin{align*}
		&\lim_{\epsilon \to 0} \int_{0}^{t} \int_{0}^\infty \zeta(u) \left[\mathcal{C}_{d}(s,u)+\mathcal{F}_{d}(s,u)\right] \Lambda_{\epsilon}(e(s,u)) \;duds \\
		&= \int_{0}^{t} \int_{0}^\infty \zeta(u) \left[\mathcal{C}_{d}(s,u)+\mathcal{F}_{d}(s,u)\right] \,\mathrm{sign}(e(s,u)) \;duds.
	\end{align*}
Consequently, taking the limit as $\epsilon \to 0$ in~\eqref{epdelmain-3} completes the proof of~\Cref{lem.diffineq}.  
\end{proof}
By virtue of \Cref{lem.diffineq}, to complete the proof of \Cref{continuousdependence result}, it remains to estimate the coagulation and fragmentation terms. We follow the classical approach, as presented in \cite[Section~8.2.5]{BLL_book} and \cite{giri2013uniqueness}, which was later adapted in \cite[Proposition 5.1]{Giri2025well}.
\begin{proof}[Proof of \Cref{continuousdependence result}]
	Fix $T>0$ and $t\in [0,T]$. First, we consider the coagulation part present in the right-hand side of~\eqref{z001}
	\begin{align*}
		&\int_{0}^{t}\int_{0}^{\infty} \zeta(u)\, \mathrm{sign}(e(s,u)) \mathcal{C}_{d}(s,u) \;duds\nonumber\\
		& \quad = \frac{1}{2} \int_{0}^{t}\int_{0}^{\infty}\int_{0}^{\infty} \Upsilon(u,u_{1}) W(s,u,u_{1}) (\xi_1+\xi_2)(s,u) e(s,u_{1}) \;du_{1}duds,
	\end{align*}
	where
	\begin{align*}
		W(s,u,u_{1}) & = \zeta(u+u_{1})\, \mathrm{sign}(e(s,u+u_{1})) - \zeta(u) \,\mathrm{sign}(e(s,u)) - \zeta(u_{1}) \,\mathrm{sign}(e(s,u_{1})). 
	\end{align*}
Since
	\begin{align*}
		W(s,u,u_{1}) e(s,u_{1}) &\leq \left( 1 + u + u_{1}\right) |e(s,u_{1})| + \left(1+u\right) |e(s,u_{1})| -  \left(1+u_{1}\right) |e(s,u_{1})| \\
		&=\left(1+2u\right) |e(s,u_{1})|,
	\end{align*}
	as a result, using \eqref{multiplicativecoagassum}, we deduce the following estimate
\begin{align}
		&\int_{0}^{t}\int_{0}^{\infty} \zeta(u)\, \mathrm{sign}(e(s,u)) \mathcal{C}_{d}(s,u) \;duds\nonumber\\
		& \quad \leq \frac{1}{2}\int_{0}^{t}\int_{0}^{\infty}\int_{0}^{\infty} \Upsilon(u,u_{1}) \left(1+2u\right) (\xi_1+\xi_2)(s,u) |e(s,u_{1})| \;du_{1}duds \nonumber\\
		& \quad \leq \frac{\Upsilon_{0}}{2} \int_{0}^{t}\int_{0}^{\infty}\int_{0}^{\infty} (1+u)(1+u_{1}) \left(1+2u\right) (\xi_1+\xi_2)(s,u) |e(s,u_{1})| \;du_{1}duds\nonumber\\
		&\quad \leq \frac{\Upsilon_{0}}{2} \int_0^t  \left[ M_{0}((\xi_1+\xi_2)(s)) + 3M_{1}((\xi_1+\xi_2)(s)) + 2M_2((\xi_1+\xi_2)(s)) \right] \int_{0}^{\infty} \zeta(u) |e(t,u)| \;duds,\label{Contributionbycoag}
\end{align}
for each $t\in [0,T]$. Next, we estimate the fragmentation term appearing in the last expression on the right-hand side of~\eqref{z001}. Using the properties of sign function, \eqref{particlebound}, \eqref{lmassconservation}, \eqref{alphaassum}, and Fubini's theorem, we infer that 
\begin{align}
	&\int_{0}^{t}\int_{0}^{\infty} \zeta(u)\, \mathrm{sign}(e(s,u)) \mathcal{F}_{d}(s,u) \;duds\nonumber\\
	& \quad = -\int_{0}^{t}\int_{0}^{\infty} \zeta(u)\alpha(u)|e(s,u)|\;duds + \int_{0}^{t}\int_{0}^{\infty}\zeta(u)\, \mathrm{sign}(e(s,u))\left(\int_{u}^{\infty}\alpha(u_{1})\beta(u|u_{1})e(s,u_{1})\;du_{1}\right) \;duds\nonumber\\
	&\quad\leq -\int_{0}^{t}\int_{0}^{\infty} \zeta(u)\alpha(u)|e(s,u)|\;duds +  \int_{0}^{t}\int_{0}^{\infty} \alpha(u)\left(n(u)+\zeta(u)-1\right)|e(s,u)|\;duds\nonumber\\
	&\quad\leq \max\left\{P_{\alpha}, Q_{\alpha}\right\}(M-1)\int_{0}^{t}\int_{0}^{\infty} \zeta(u)|e(s,u)|\;duds,\label{Contributionbyfrag}
\end{align}
for each $t\in [0,T]$. By inserting the previously derived inequalities \eqref{Contributionbycoag} and \eqref{Contributionbyfrag} into~\eqref{z001}, we obtain 
\begin{equation*}  
	\nu(t) \leq \nu(0) + R_{0} \int_{0}^{t} \left[ 1+ M_{0}((\xi_1+\xi_2)(s)) + 3M_{1}((\xi_1+\xi_2)(s)) + 2M_2((\xi_1+\xi_2)(s)) \right] \nu(s) \, ds,  
\end{equation*}  
for each $t \in [0,T]$ with $R_{0} := L_{0}+\Upsilon_{0}/2 + \max\left\{P_{\alpha}, Q_{\alpha}\right\}(M-1)$. Finally, using~\eqref{z014} along with Gronwall’s lemma, we arrive at the desired conclusion, thereby completing the proof of \Cref{continuousdependence result}. 
\end{proof}

\begin{proof}[Proof of \Cref{uniqueness theorem}]   
	Consider two weak solutions $\xi_1$ and $\xi_2$ of \eqref{maineq.1}--\eqref{maineq.2}, as defined in~\Cref{definitionofsolution}, corresponding to the initial data $\xi_{1}^{{\mathrm{in}}}$ and $\xi_{2}^{{\mathrm{in}}}$, respectively. Suppose that $\xi_{1}^{{\mathrm{in}}} = \xi_{2}^{{\mathrm{in}}}$ almost everywhere in $(0, \infty)$. Then, applying \Cref{continuousdependence result}, we obtain  
	\begin{equation*}  
		\int_{0}^{\infty} \left( 1 + u\right) |\xi_{1}(t,u) - \xi_{2}(t,u)| \,du = 0,  
	\end{equation*}  
	which directly implies that $\xi_1(t) = \xi_2(t)$ for all $t \in [0,T]$, thereby completing the proof of \Cref{uniqueness theorem}.  
\end{proof}

\section{Large time behaviour}
\noindent We set $\xi^{{\mathrm{in}}}\in Y_{+}$ with $\xi^{{\mathrm{in}}}\not\equiv 0$ and denote by $\xi$ a weak solution to \eqref{maineq.1}--\eqref{maineq.2}, as defined in \Cref{definitionofsolution}, throughout this section. In the absence of fragmentation, the growth-coagulation equations \eqref{maineq.1}--\eqref{maineq.2} govern only the processes of coagulation, growth, birth, and death. Consequently, both the total number of particles, $M_{0}(\xi(\cdot))$, and the total mass, $M_{1}(\xi(\cdot))$, in the system are expected to decrease over time due to the condition \eqref{deathdominated}. This leads to the following result motivated by~\cite[Proposition 4.1]{lachowicz2003oort}.
\begin{proposition}\label{monotonicdecresingmoments}
	Suppose that $\Upsilon$ and $\mu$ satisfy \eqref{multiplicativecoagassum} and \eqref{mu}, respectively. Assume that $\alpha \equiv 0$ and $\beta \equiv 0$, and let $g$, $a$, and $\xi^{{\mathrm{in}}}$ be parameters satisfying \eqref{growthassum}--\eqref{initialassum}. Moreover, suppose that \eqref{deathdominated} holds. Consider $\xi$ as a weak solution to \eqref{maineq.1}--\eqref{maineq.2} as defined in \Cref{definitionofsolution}. Then, for $m=0, 1$ and $t_{2}\ge t_{1}\ge 0$, we have 
	\begin{align}\label{monotonicityofmomentsintime}
		\int_{0}^{\infty}u^{m}\xi(t_{2},u)\;du \leq 	\int_{0}^{\infty}u^{m}\xi(t_{1},u)\;du. 
	\end{align}
\end{proposition}
\begin{proof}
We set $ \vartheta(u) = \min \{u, K\}$ for $u\in [0,\infty)$ and $K>0$. Inserting $\vartheta$ in \eqref{weak formulation} at $t=t_1$ and $t=t_2$, with $\alpha \equiv 0$ and $\beta \equiv 0$, we derive the following inequality 
\begin{align}\label{LargeM-3} 
	\int_{0}^{K} u \xi(t_{2},u)\;du &\leq \int_{0}^{\infty} u \xi(t_{1},u)\;du   
	+\int_{t_{1}}^{t_{2}}\int_{0}^{K} \left(g(u)-u\mu(u)\right)\xi(s,u)\;duds \nonumber \\  
	&\quad +\frac{1}{2}\int_{t_{1}}^{t_{2}}\int_{0}^{\infty} \int_{0}^{\infty} \tilde{\vartheta}(u,u_{1}) \Upsilon(u,u_{1})\xi(s,u) \xi(s,u_{1})\;du du_{1}ds. 
\end{align}  
Using the conditions  
\begin{align*}  
	\begin{cases}  
		\tilde{\vartheta}(u,u_{1})=-K &\quad \text{for } u,u_{1}\geq K, \\  
		\tilde{\vartheta}(u,u_{1})\leq 0 &\quad \text{otherwise},  
	\end{cases}  
\end{align*} 
along with \eqref{deathdominated} in \eqref{LargeM-3}, we obtain  
\begin{align}\label{LargeM-4}  
	\int_{0}^{K} u \xi(t_{2},u)\;du &\leq \int_{0}^{\infty} u \xi(t_{1},u)\;du  -\frac{K}{2}\int_{t_{1}}^{t_{2}}\int_{K}^{\infty} \int_{K}^{\infty} \Upsilon(u,u_{1})\xi(s,u) \xi(s,u_{1})\;du du_{1}ds.  
\end{align}
Next, by utilizing the non-negativity of $K$, $\Upsilon$, and $\xi$, we infer that
\begin{align}
	\int_{0}^{K} u\xi(t_{2},u)\;du  &\leq   	\int_{0}^{\infty} u\xi(t_{1},u)\;du.\nonumber
\end{align}
Taking the limit as $K \rightarrow \infty$ and applying Lebesgue's DCT, we obtain the stated result~\eqref{monotonicityofmomentsintime} for $m=1$. Next, to check the validity of \eqref{monotonicityofmomentsintime} for $m=0$, we consider
\begin{align*}
	\vartheta(u)=\vartheta_{\epsilon}(u)=\begin{cases}
		\frac{u}{\epsilon},\quad &u\in [0,\epsilon),\\
		1, \quad &u\in [\epsilon, \infty), 
	\end{cases}
\end{align*} 
in~\eqref{weak formulation} with $t=t_{1}$ and $t=t_{2}$. Since $\alpha \equiv 0$ and $\beta \equiv 0$, by subtracting the resulting expressions, we obtain
\begin{align*}
	\int_{0}^{\infty} \xi(t_{2},u)\vartheta_{\epsilon}(u)\;du  &=   	\int_{0}^{\infty} \xi(t_{1},u)\vartheta_{\epsilon}(u)\;du   
	+\int_{t_{1}}^{t_{2}}\int_{0}^{\infty} \vartheta_{\epsilon}^{\prime}(u)g(u)\xi(s,u)\;duds\nonumber \\ &+\vartheta_{\epsilon}(0)\int_{t_{1}}^{t_{2}}\int_{0}^{\infty} a(u)\xi(s,u)\;duds-\int_{t_{1}}^{t_{2}}\int_{0}^{\infty} \mu(u)\vartheta_{\epsilon}(u)\xi(s,u)\;duds \nonumber \\
	&+\frac{1}{2}\int_{t_{1}}^{t_{2}}\int_{0}^{\infty} \int_{0}^{\infty}\tilde{\vartheta_{\epsilon}}(u,u_{1}) \Upsilon(u,u_{1})\xi(s,u) \xi(s,u_{1})\;du du_{1}ds.
\end{align*}
 Using $\vartheta_{\epsilon}(0)=0$ and \eqref{deathdominated}, we deduce the following inequality
\begin{align}\label{Largezeromomentvanishing}
	\int_{0}^{\infty} \xi(t_{2},u)\vartheta_{\epsilon}(u)\;du  &\leq    	\int_{0}^{\infty} \xi(t_{1},u)\vartheta_{\epsilon}(u)\;du   +\int_{t_{1}}^{t_{2}}\int_{0}^{\epsilon} \mu(u)\xi(s,u)\;duds\nonumber\\
	&+\frac{1}{2}\int_{t_{1}}^{t_{2}}\int_{0}^{\infty} \int_{0}^{\infty}\tilde{\vartheta_{\epsilon}}(u,u_{1}) \Upsilon(u,u_{1})\xi(s,u) \xi(s,u_{1})\;du du_{1}ds.
\end{align}
Observe that the mapping  
$s \mapsto \int_{0}^{\infty} \int_{0}^{\infty} \Upsilon(u,u_{1})\xi(s,u) \xi(s,u_{1}) \,du \,du_{1}$
belongs to $L^{1}(t_{1}, t_{2})$. Moreover, $ |\tilde{\vartheta_{\epsilon}}(u,u_{1})| \leq 3$, $ \lim_{\epsilon\to 0}\vartheta_{\epsilon}(u) = 1$, and $|\vartheta_{\epsilon}(u)| \leq 1$ for $u \in (0,\infty)$. Additionally, we have $\xi \in L^{\infty}(0,T; Y_{+})$ for any $T>0$ and $\mu \xi \in L^{1}([t_{1}, t_{2}]; L^{1}([0,\infty)))$ by \Cref{definitionofsolution} and \eqref{mu}. Hence, by taking the limit as $ \epsilon \to 0$ in \eqref{Largezeromomentvanishing} and applying Lebesgue's DCT, we obtain  
\begin{align}\label{Largezeromomentvanishing-1}
	\int_{0}^{\infty} \xi(t_{2},u)\;du  \leq \int_{0}^{\infty} \xi(t_{1},u)\;du -\frac{1}{2} \int_{t_{1}}^{t_{2}} \int_{0}^{\infty} \int_{0}^{\infty} \Upsilon(u,u_{1})\xi(s,u) \xi(s,u_{1})\;dudu_{1}ds.
\end{align}
This confirms that the stated result~\eqref{monotonicityofmomentsintime} also holds for $m=0$, i.e.,
\begin{align}
	\int_{0}^{\infty} \xi(t_{2},u)\;du  &\leq    	\int_{0}^{\infty} \xi(t_{1},u)\;du.\nonumber
\end{align}
Hence, the proof of \Cref{monotonicdecresingmoments} is complete. 
\end{proof} 
We next proceed to prove \Cref{Largetimebehavior}.
\begin{proof}[\textbf{Proof of \Cref{Largetimebehavior}(a)}]
 For $\mathcal{V}>0$, $\epsilon \in (0,1)\cap (0, \mathcal{V})$, and $t\ge 0$, we define 
\begin{align*}
	G(t,\mathcal{V})=\int_{0}^{\mathcal{V}}\xi(t,u) \;du
\end{align*} 
and 
\begin{align}
\vartheta_{\epsilon}(u)=\begin{cases}
	\frac{u}{\epsilon} &\quad\text{ for }\;0\leq u<\epsilon,\\
	    1 &\quad \text{ for } \epsilon\leq u\leq \mathcal{V},\\
	    \frac{\mathcal{V}+\epsilon-u}{\epsilon}&\quad \text{ for } \mathcal{V}< u\leq \mathcal{V}+\epsilon,\\
		0&\quad \text{ for } u> \mathcal{V}+\epsilon.
	\end{cases}
\end{align}
Then we have $\vartheta_{\epsilon}(0)=0$, $\vartheta_{\epsilon}\in W^{1,\infty}([0, \infty))$, $\vartheta_{\epsilon}^{\prime}(u)\leq 0$ for $u\in (\epsilon, \infty)$, and $|\vartheta_{\epsilon}(u)|\leq 1$ for $u\in (0,\infty)$. Since $\alpha \equiv 0$ and $\beta \equiv 0$, using \eqref{weak formulation} at $t=t_{1}$ and $t=t_{2}$ with $\vartheta=\vartheta_{\epsilon}$, we deduce that 
\begin{align*}
	\int_{0}^{\infty} \xi(t_{2},u)\vartheta_{\epsilon}(u)\;du  &\leq    	\int_{0}^{\infty} \xi(t_{1},u)\vartheta_{\epsilon}(u)\;du
	+\frac{1}{\epsilon}\int_{t_{1}}^{t_{2}}\int_{0}^{\epsilon} g(u)\xi(s,u)\;duds\nonumber \\ 
	&\quad +\frac{1}{2}\int_{t_{1}}^{t_{2}}\int_{0}^{\infty} \int_{0}^{\infty}\tilde{\vartheta_{\epsilon}}(u,u_{1}) \Upsilon(u,u_{1})\xi(s,u) \xi(s,u_{1})\;du du_{1}ds,
\end{align*}
for $t_{2}\ge t_{1}\ge 0$. Next, by using \eqref{deathdominated}, we infer that 
\begin{align}\label{Large-1}
	\int_{0}^{\infty} \xi(t_{2},u)\vartheta_{\epsilon}(u)\;du  &\leq    	\int_{0}^{\infty} \xi(t_{1},u)\vartheta_{\epsilon}(u)\;du   +\int_{t_{1}}^{t_{2}}\int_{0}^{\epsilon} \mu(u)\xi(s,u)\;duds\nonumber \\ 
	&\quad +\frac{1}{2}\int_{t_{1}}^{t_{2}}\int_{0}^{\infty} \int_{0}^{\infty}\tilde{\vartheta_{\epsilon}}(u,u_{1}) \Upsilon(u,u_{1})\xi(s,u) \xi(s,u_{1})\;du du_{1}ds.
\end{align}
Note that the mapping $s\mapsto \int_{0}^{\infty} \int_{0}^{\infty} \Upsilon(u,u_{1})\xi(s,u) \xi(s,u_{1})\;du du_{1}$ is in $L^{1}(t_{1}, t_{2})$ with $|\tilde{\vartheta_{\epsilon}}(u,u_{1})|\leq 3$, $\lim_{\epsilon\to 0}\vartheta_{\epsilon}(u)=\chi_{(0,\mathcal{V}]}(u)$, $|\vartheta_{\epsilon}(u)|\leq 1$, $u\in (0,\infty)$. Furthermore, we have $\xi \in L^{\infty}(0,T; Y_{+})$ for any $T>0$ and $\mu \xi\in L^{1}([t_{1}, t_{2}); L^{1}([0,\infty)))$ by \Cref{definitionofsolution} and \eqref{mu}. Therefore, by taking the limit as $\epsilon\to 0$ in \eqref{Large-1} and applying Lebesgue's DCT, we obtain   
\begin{align}\label{Large-2}
	G(t_{2},\mathcal{V})-G(t_{1},\mathcal{V})\leq -\frac{1}{2}\int_{t_{1}}^{t_{2}}\int_{0}^{\mathcal{V}} \int_{0}^{\mathcal{V}} \Upsilon(u,u_{1})\xi(s,u) \xi(s,u_{1})\;du du_{1}ds.
\end{align}
A direct consequence of \eqref{Large-2} is that $G(\cdot, \mathcal{V})$ remains non-negative and decreasing over time. Consequently, there exists a limit $G(\mathcal{V}) \geq 0$ such that  
\begin{align}\label{Large-3}  
	\lim_{t \to \infty} G(t,\mathcal{V}) = G(\mathcal{V}).  
\end{align}  
Additionally, from \eqref{Large-2} and the lower bound in \eqref{coaglowerbound}, we deduce that for any $\theta \in (0, \mathcal{V})$, the following holds:  
\begin{align*}  
	\int_{0}^{t} \left(\int_{\theta}^{\mathcal{V}} \xi(s,u) \,du \right)^{2}\,ds &\leq \frac{1}{\delta_{\theta}} \int_{0}^{t} \int_{0}^{\mathcal{V}} \int_{0}^{\mathcal{V}} \Upsilon(u,u_{1}) \xi(s,u) \xi(s,u_{1}) \,du \,du_{1} \,ds \\  
	&\leq \frac{2}{\delta_{\theta}} G(0, \mathcal{V}).  
\end{align*}  
This implies that  
\begin{align}\label{Large-4}  
	t \mapsto G(t,\mathcal{V}) - G(t,\theta) \in L^{2}(0,\infty).  
\end{align}  
From \eqref{Large-3}, \eqref{Large-4}, and the fact that $G(\cdot, \theta)$ is non-increasing with respect to time, we conclude that  
\begin{align*}  
	0 \leq G(\mathcal{V}) = G(\theta) \leq G(0, \theta), \quad \theta \in (0,\mathcal{V}).  
\end{align*}  
Since the initial data $\xi^{{\mathrm{in}}}$ belongs to $L^{1}([0,\infty))$, taking the limit as $\theta \to 0$ in the above expression leads to $G(\mathcal{V}) = 0$ for all $\mathcal{V}> 0$. Furthermore, by \eqref{monotonicityofmomentsintime}, we observe that  
\begin{align*}  
	M_{0}(\xi(t)) \leq G(t, \mathcal{V}) + \frac{1}{\mathcal{V}} \int_{\mathcal{V}}^{\infty} u \xi(t,u) \,du \leq G(t,\mathcal{V}) + \frac{1}{\mathcal{V}} \int_{0}^{\infty} u \xi^{{\mathrm{in}}}(u) \,du,  
\end{align*}  
for any $t>0$. Taking the limit as $t \to \infty$ and subsequently $\mathcal{V} \to \infty$ in the above inequality establishes \eqref{vanishingzerothmoment}, thereby completing the proof of \Cref{Largetimebehavior}(a).    	
\end{proof}
We now proceed to prove \Cref{Largetimebehavior}(b), beginning with a key intermediate result.
\begin{lemma}\label{LargeM-0}
	If all the hypotheses made in \Cref{monotonicdecresingmoments} are satisfied, then for $t_{2}\ge t_{1}\ge 0$, we have 
	\begin{align}\label{LargeM-1}
		\int_{t_{1}}^{t_{2}}\int_{0}^{\infty}\int_{0}^{\infty}\Upsilon(u,u_{1})\xi(s,u) \xi(s,u_{1})\;dudu_{1}ds \leq 2\int_{0}^{\infty}\xi(t_{1},u)\;du
	\end{align}
and 
\begin{align}\label{LargeM-2}
	\int_{t_{1}}^{t_{2}}\int_{K}^{\infty}\int_{K}^{\infty}\Upsilon(u,u_{1})\xi(s,u) \xi(s,u_{1})\;dudu_{1}ds \leq \frac{2}{K}\int_{0}^{\infty}u\xi(t_{1},u)\;du,
\end{align}
for any $K>0$.
\end{lemma}
\begin{proof}
The inequalities \eqref{LargeM-1} and \eqref{LargeM-2} follow directly from \eqref{Largezeromomentvanishing-1} and \eqref{LargeM-4}, respectively.   
\end{proof}
Following the approach outlined in \cite[Theorem 1.1]{escobedo2002gelation} and \cite[Theorem 2.5]{lachowicz2003oort}, we now proceed to complete the proof of \Cref{Largetimebehavior}(b).
\begin{proof}[\textbf{Proof of \Cref{Largetimebehavior}(b)}]
First, let $\lambda \in (1,2)$ and define $\eta(u) = \left(u^{1-\lambda/2} - 1\right)_{+}$ for $u \in (0, \infty)$. Since $\lambda > 1$, it follows that  
\begin{align*}  
	\mathcal{N}:= \int_{0}^{\infty} u^{-1/2} \eta^{\prime}(u) \;du < \infty.  
\end{align*}  
For $t_{2} \geq t_{1} \geq 0$, applying Hölder’s inequality along with \eqref{coaglowerboundM} and \eqref{LargeM-2}, we establish the following bound 
\begin{align*}  
	&\int_{t_{1}}^{t_{2}}\left( \int_{0}^{\infty} \eta^{\prime}(K) \int_{K}^{\infty} u^{\lambda/2} \xi(s,u) \;du dK \right)^{2}\;ds \\  
	&\leq \mathcal{N} \int_{t_{1}}^{t_{2}} \int_{0}^{\infty} \eta^{\prime}(K) K^{1/2} \left( \int_{K}^{\infty} u^{\lambda/2} \xi(s,u) \;du \right)^{2} \;dK ds \\  
	&\leq \frac{\mathcal{N}}{\Upsilon_{2}} \int_{0}^{\infty} \eta^{\prime}(K) K^{1/2} \int_{t_{1}}^{t_{2}} \int_{K}^{\infty} \int_{K}^{\infty} \Upsilon(u,u_{1}) \xi(s,u) \xi(s,u_{1}) \;du du_{1} ds dK \\  
	&\leq \frac{2{\mathcal{N}}^{2}}{\Upsilon_{2}} M_{1}(\xi(t_{1})).  
\end{align*}  
Next, we observe that  
\begin{align*}  
	\int_{0}^{\infty} \eta^{\prime}(K) \int_{K}^{\infty} u^{\lambda/2} \xi(s,u) \;du dK &= \int_{0}^{\infty} u^{\lambda/2} \eta(u) \xi(s,u) \;du \\  
	&\geq \left(2^{\frac{2-\lambda}{2}} - 1\right) \int_{2}^{\infty} u \xi(s,u) \;du.  
\end{align*}  
Combining these inequalities, we obtain  
\begin{align}\label{TempLarge-1}  
	\int_{t_{1}}^{t_{2}} \left( \int_{2}^{\infty} u \xi(s,u) \;du \right)^{2} \;ds \leq C(\lambda,\Upsilon_{2}) M_{1}(\xi(t_{1})),  
\end{align}  
for some positive constant $C(\lambda,\Upsilon_{2}) $ that depends only on $\lambda$ and $\Upsilon_{2}$. Next, applying \eqref{coaglowerboundM} and \eqref{LargeM-1}, we obtain  
\begin{align}\label{TempLarge-2}  
	\int_{t_{1}}^{t_{2}} \left( \int_{0}^{2} u \xi(s,u)\;du \right)^{2} \;ds \leq C(\Upsilon_{2}) \int_{0}^{\infty} \xi(t_{1},u)\;du,  
\end{align}  
for some positive constant $C(\Upsilon_{2})$. Furthermore, using Young’s inequality, \eqref{monotonicityofmomentsintime}, \eqref{TempLarge-1}, and \eqref{TempLarge-2}, we infer that  
\begin{align}\label{TempLarge-3}  
	\int_{t}^{\infty} M_{1}^{2}(\xi(s)) \;ds \leq 2C\|\xi^{{\mathrm{in}}}\|, \quad t\geq 0,  
\end{align}  
for some positive constant $C$ that depends only on $\lambda$ and $\Upsilon_{2}$. Recalling the monotonicity property of \eqref{monotonicityofmomentsintime}, we note that the total mass $M_{1}(\xi(\cdot))$ is non-increasing. Consequently, by \eqref{TempLarge-3}, it follows that  
\begin{align*}  
	tM_{1}^{2}(\xi(t)) \leq \int_{0}^{t} M_{1}^{2}(\xi(s)) \;ds \leq 2C\|\xi^{{\mathrm{in}}}\|,\quad t\ge 0.  
\end{align*}  
This leads to the following asymptotic decay estimate for the total mass $M_{1}(\xi(t))$, i.e.,  
\begin{align*}  
	M_{1}(\xi(t)) \leq \frac{\sqrt{2C\|\xi^{{\mathrm{in}}}\|}}{\sqrt{t}}, \quad t > 0.  
\end{align*}  
As a direct consequence, we obtain the limit 
\begin{align*}
	\lim_{t\rightarrow \infty} M_{1}(\xi(t)) = 0.
\end{align*}
The remaining case to consider is $\lambda=2$. In this scenario, \eqref{TempLarge-3} follows directly from \eqref{coaglowerboundM} and \eqref{LargeM-1}. The proof of \Cref{Largetimebehavior}(b) is then completed using the same reasoning as above.
\end{proof}

	%%%%%%%%%%%%%%%%%%%%%%%%%%%%%%%%

	\noindent \textbf{Acknowledgments.} SS's research was funded by the Council of Scientific and Industrial Research, India, through grant No.09/143(0987)/2019-EMR-I.
	%%%%%%%%%%%%%%%%%%%%%%%%%%%%%%%%%%%%%%%%%%%%%%%%
	
	%%%%%%%%%%%%%%%%
	%%%%%%%%%%%%%%%%
	\bibliographystyle{abbrv}
	\bibliography{RenewalGrowth_SS_AKG}

\begin{thebibliography}{10}

\bibitem{ackleh1997parameter}
A.~Ackleh.
\newblock Parameter estimation in a structured algal coagulation-fragmentation
  model.
\newblock {\em Nonlinear Analysis}, 28(5):837--854, 1997.

\bibitem{ackleh2003first}
A.~S. Ackleh and K.~Deng.
\newblock On a first-order hyperbolic coagulation model.
\newblock {\em Mathematical methods in the applied sciences}, 26(8):703--715,
  2003.

\bibitem{ackleh1997modeling}
A.~S. Ackleh and B.~G. Fitzpatrick.
\newblock Modeling aggregation and growth processes in an algal population
  model: analysis and computations.
\newblock {\em Journal of Mathematical Biology}, 35:480--502, 1997.

\bibitem{ackleh2021structured}
A.~S. Ackleh, R.~Lyons, and N.~Saintier.
\newblock A structured coagulation-fragmentation equation in the space of radon
  measures: Unifying discrete and continuous models.
\newblock {\em ESAIM: Mathematical Modelling and Numerical Analysis},
  55(5):2473--2501, 2021.

\bibitem{ackleh2023finite}
A.~S. Ackleh, R.~Lyons, and N.~Saintier.
\newblock Finite difference schemes for a size structured
  coagulation-fragmentation model in the space of radon measures.
\newblock {\em IMA Journal of Numerical Analysis}, 43(6):3357--3395, 2023.

\bibitem{ackleh2023high}
A.~S. Ackleh, R.~Lyons, and N.~Saintier.
\newblock High resolution finite difference schemes for a size structured
  coagulation-fragmentation model in the space of radon measures.
\newblock {\em Mathematical Biosciences and Engineering}, 20(7):11805--11820,
  2023.

\bibitem{adams2003sobolev}
R.~A. Adams and J.~J. Fournier.
\newblock {\em Sobolev spaces}, volume 140.
\newblock Elsevier, 2003.

\bibitem{Adler1997}
R.~Adler.
\newblock Superprocesses and plankton dynamics.
\newblock Technical report, North Carolina Univ. at Chapel Hill, Dept. of
  Statistics, 1997.

\bibitem{aldous1999deterministic}
D.~J. Aldous.
\newblock Deterministic and stochastic models for coalescence (aggregation and
  coagulation): A review of the mean-field theory for probabilists.
\newblock {\em Bernoulli}, pages 3--48, 1999.

\bibitem{alldredge1988situ}
A.~L. Alldredge and C.~Gotschalk.
\newblock In situ settling behavior of marine snow 1.
\newblock {\em Limnology and Oceanography}, 33(3):339--351, 1988.

\bibitem{alldredge1989direct}
A.~L. Alldredge and C.~Gotschalk.
\newblock Direct observations of the mass flocculation of diatom blooms:
  characteristics, settling velocities and formation of diatom aggregates.
\newblock {\em Deep Sea Research Part A. Oceanographic Research Papers},
  36(2):159--171, 1989.

\bibitem{banasiak2011blow}
J.~Banasiak.
\newblock Blow-up of solutions to some coagulation and fragmentation equations
  with growth.
\newblock {\em Discrete and Continuous Dynamical Systems (Dynamical systems,
  differential equations and applications. 8th AIMS Conference. Suppl. vol.
  I)}, pages 126--134, 2011.

\bibitem{Banasiak2012}
J.~Banasiak.
\newblock Transport processes with coagulation and strong fragmentation.
\newblock {\em Discrete and Continuous Dynamical Systems. Series B. A Journal
  Bridging Mathematics and Sciences}, 17(2):445--472, 2012.

\bibitem{banasiak2009coagulation}
J.~Banasiak and W.~Lamb.
\newblock Coagulation, fragmentation and growth processes in a size structured
  population.
\newblock {\em Discrete and Continuous Dynamical Systems-Series B},
  11(3):563--585, 2009.

\bibitem{banasiak2020growth}
J.~Banasiak and W.~Lamb.
\newblock Growth--fragmentation--coagulation equations with unbounded
  coagulation kernels.
\newblock {\em Philosophical Transactions of the Royal Society A},
  378(2185):20190612, 2020.

\bibitem{BLL_book}
J.~Banasiak, W.~Lamb, and {\relax Ph}.~Lauren{\c{c}}ot.
\newblock {\em Analytic Methods for Coagulation-Fragmentation Models}.
\newblock Chapman and Hall/CRC, 2019.

\bibitem{banasiak2022explicit}
J.~Banasiak, D.~W. Poka, and S.~Shindin.
\newblock Explicit solutions to some fragmentation equations with growth or
  decay.
\newblock {\em Journal of Physics A: Mathematical and Theoretical},
  55(19):194001, 2022.

\bibitem{banasiak2024growth}
J.~Banasiak, D.~W. Poka, and S.~Shindin.
\newblock Growth-fragmentation equations with mckendrick--von foerster boundary
  condition.
\newblock {\em Discrete and Continuous Dynamical Systems-S},
  17(5\&6):2030--2057, 2024.

\bibitem{barik2022mass}
P.~K. Barik, P.~Rai, and A.~K. Giri.
\newblock Mass-conserving weak solutions to oort-hulst-safronov coagulation
  equation with singular rates.
\newblock {\em Journal of Differential Equations}, 326:164--200, 2022.

\bibitem{canizo2021spectral}
J.~A. Ca{\~n}izo, P.~Gabriel, and H.~Yoldas.
\newblock Spectral gap for the growth-fragmentation equation via harris's
  theorem.
\newblock {\em SIAM Journal on Mathematical Analysis}, 53(5):5185--5214, 2021.

\bibitem{cushing1983fish}
D.~Cushing.
\newblock Are fish larvae too dilute to affect the density of their food
  organisms?
\newblock {\em Journal of Plankton Research}, 5(6):847--854, 1983.

\bibitem{cushing1990plankton}
D.~Cushing.
\newblock Plankton production and year-class strength in fish populations: an
  update of the match/mismatch hypothesis.
\newblock In {\em Advances in marine biology}, volume~26, pages 249--293.
  Elsevier, 1990.

\bibitem{diperna1989ordinary}
R.~J. DiPerna and P.-L. Lions.
\newblock Ordinary differential equations, transport theory and sobolev spaces.
\newblock {\em Inventiones mathematicae}, 98(3):511--547, 1989.

\bibitem{edwards2012functional}
R.~E. Edwards.
\newblock {\em Functional analysis: theory and applications}.
\newblock Courier Corporation, 2012.

\bibitem{escobedo2003gelation}
M.~Escobedo, {\relax Ph}.~Lauren{\c{c}}ot, S.~Mischler, and B.~Perthame.
\newblock Gelation and mass conservation in coagulation-fragmentation models.
\newblock {\em Journal of Differential Equations}, 195(1):143--174, 2003.

\bibitem{escobedo2002gelation}
M.~Escobedo, S.~Mischler, and B.~Perthame.
\newblock Gelation in coagulation and fragmentation models.
\newblock {\em Communications in Mathematical Physics}, 231:157--188, 2002.

\bibitem{EMRR2005}
M.~Escobedo, S.~Mischler, and M.~Rodriguez~Ricard.
\newblock On self-similarity and stationary problem for fragmentation and
  coagulation models.
\newblock {\em Annales de l'Institut Henri Poincar\'e. Analyse Non Lin\'eaire},
  22(1):99--125, 2005.

\bibitem{fonseca2007modern}
I.~Fonseca and G.~Leoni.
\newblock {\em Modern methods in the calculus of variations: $L^{p}$ spaces}.
\newblock Springer Science \& Business Media, 2007.

\bibitem{Friedlander2000}
S.~K. Friedlander.
\newblock {\em Smoke, dust, and haze. {F}undamentals of aerosol dynamics. $2$nd
  edition}.
\newblock Oxford Univ. Press, 2000.

\bibitem{FriedmanReitich1990}
A.~Friedman and F.~Reitich.
\newblock A hyperbolic inverse problem arising in the evolution of combustion
  aerosol.
\newblock {\em Archive for Rational Mechanics and Analysis}, 110(4):313--350,
  1990.

\bibitem{Gajewski1983}
H.~Gajewski.
\newblock On a first order partial differential equation with nonlocal
  nonlinearity.
\newblock {\em Mathematische Nachrichten}, 111:289--300, 1983.

\bibitem{GajewskiZacharias1982}
H.~Gajewski and K.~Zacharias.
\newblock On an initial value problem for a coagulation equation with growth
  term.
\newblock {\em Mathematische Nachrichten}, 109:135--156, 1982.

\bibitem{giri2010mathematical}
A.~K. Giri.
\newblock {\em Mathematical and numerical analysis for
  coagulation-fragmentation equations}.
\newblock PhD thesis, Otto-von-Guericke-Universit{\"a}t Magdeburg, 2010.

\bibitem{giri2013uniqueness}
A.~K. Giri.
\newblock On the uniqueness for coagulation and multiple fragmentation
  equation.
\newblock {\em Kinetic \& Related Models}, 6(3), 2013.

\bibitem{Giri2021weak}
A.~K. Giri and {\relax Ph}.~Lauren{\c{c}}ot.
\newblock Weak solutions to the collision-induced breakage equation with
  dominating coagulation.
\newblock {\em Journal of Differential Equations}, 280:690--729, 2021.

\bibitem{Giri2025well}
A.~K. Giri, {\relax Ph}.~Lauren{\c{c}}ot, and S.~Si.
\newblock Well-posedness of the growth-coagulation equation with singular
  kernels.
\newblock {\em Nonlinear Analysis: Real World Applications}, 84:104300, 2025.

\bibitem{hjort1914fluctuations}
J.~Hjort.
\newblock Fluctuations in the great fisheries of northern europe viewed in the
  light of biological research.
\newblock {\em Rapports et procès-verbaux des réunions}, 20:1--228, 1914.

\bibitem{jackson1990model}
G.~A. Jackson.
\newblock A model of the formation of marine algal flocs by physical
  coagulation processes.
\newblock {\em Deep Sea Research Part A. Oceanographic Research Papers},
  37(8):1197--1211, 1990.

\bibitem{lachowicz2003oort}
M.~Lachowicz, {\relax Ph}.~Lauren{\c{c}}ot, and D.~Wrzosek.
\newblock On the oort--hulst--safronov coagulation equation and its relation to
  the smoluchowski equation.
\newblock {\em SIAM journal on mathematical analysis}, 34(6):1399--1421, 2003.

\bibitem{Laurencot2001}
{\relax Ph}.~Lauren{\c{c}}ot.
\newblock The {Lifshitz}-{Slyozov} equation with encounters.
\newblock {\em Mathematical Models and Methods in Applied Sciences},
  11(4):731--748, 2001.

\bibitem{laurenccot2001weak}
{\relax Ph}.~Lauren{\c{c}}ot.
\newblock Weak solutions to the lifshitz-slyozov-wagner equation.
\newblock {\em Indiana University mathematics journal}, pages 1319--1346, 2001.

\bibitem{laurenccot2014weak}
{\relax Ph}.~Lauren{\c{c}}ot.
\newblock Weak compactness techniques and coagulation equations.
\newblock In {\em Evolutionary equations with applications in natural
  sciences}, pages 199--253. Springer, 2014.

\bibitem{laurenccot2007well}
{\relax Ph}.~Lauren{\c{c}}ot and C.~Walker.
\newblock Well-posedness for a model of prion proliferation dynamics.
\newblock {\em Journal of Evolution Equations}, 7:241--264, 2007.

\bibitem{Le1977}
C.~H. L{\^e}.
\newblock {\em Etude de la classe des op\'erateurs {$m$}-accr\'etifs de
  {$L^{1}(\Omega )$} et accr\'etifs dans {$L^{\infty }(\Omega )$}}.
\newblock PhD thesis, Universit{\'e} de Paris {VI}, 1977.
\newblock Th{\`e}se de {$3^{\text{\`eme}}$} cycle.

\bibitem{leis2017existence}
E.~Leis and C.~Walker.
\newblock Existence of global classical and weak solutions to a prion equation
  with polymer joining.
\newblock {\em Journal of Evolution Equations}, 17:1227--1258, 2017.

\bibitem{Lissauer1993}
J.~J. Lissauer.
\newblock Planet formation.
\newblock {\em Annual Review of Astronomy and Astrophysics}, 31(1):129--172,
  1993.

\bibitem{nowicki2022quantifying}
M.~Nowicki, T.~DeVries, and D.~A. Siegel.
\newblock Quantifying the carbon export and sequestration pathways of the
  ocean's biological carbon pump.
\newblock {\em Global Biogeochemical Cycles}, 36(3):e2021GB007083, 2022.

\bibitem{perthame2006transport}
B.~Perthame.
\newblock {\em Transport equations in biology}.
\newblock Springer Science \& Business Media, 2006.

\bibitem{rudnicki2006fragmentation}
R.~Rudnicki and R.~Wieczorek.
\newblock Fragmentation-coagulation models of phytoplankton.
\newblock {\em Bulletin of the Polish Academy of Sciences. Mathematics},
  54(2):175--191, 2006.

\bibitem{Safronov1972}
V.~Safronov.
\newblock {\em Evolution of the protoplanetary cloud and formation of the earth
  and the planets}.
\newblock Israel Program for Scientific Translations, Jerusalem, 1972.

\bibitem{smetacek1985role}
V.~Smetacek.
\newblock Role of sinking in diatom life-history cycles: ecological,
  evolutionary and geological significance.
\newblock {\em Marine biology}, 84:239--251, 1985.

\bibitem{smit1994aggregation}
D.~Smit, M.~Hounslow, and W.~Paterson.
\newblock Aggregation and gelation—i. analytical solutions for cst and batch
  operation.
\newblock {\em Chemical Engineering Science}, 49(7):1025--1035, 1994.

\bibitem{Stewart1989}
I.~Stewart and E.~Meister.
\newblock A global existence theorem for the general coagulation-fragmentation
  equation with unbounded kernels.
\newblock {\em Mathematical Methods in the Applied Sciences}, 11(5):627--648,
  1989.

\bibitem{Stewart1990}
I.~W. Stewart.
\newblock A uniqueness theorem for the coagulation-fragmentation equation.
\newblock {\em Mathematical Proceedings of the Cambridge Philosophical
  Society}, 107(3):573--578, 1990.

\bibitem{vigil1989stability}
R.~D. Vigil and R.~M. Ziff.
\newblock On the stability of coagulation-fragmentation population balances.
\newblock {\em Journal of Colloid and Interface Science}, 133(1):257--264,
  1989.

\bibitem{Vrabie2003}
I.~I. Vrabie.
\newblock {\em Compactness Methods for Nonlinear Evolutions}, volume~75 of {\em
  Pitman Monographs and Surveys in Pure and Applied Mathematics}.
\newblock CRC Press/Longman Scientific \& Technical, Harlow., 1995.

\bibitem{Wells2018}
J.~Wells.
\newblock {\em Modelling coagulation in industrial spray drying: {A}n efficient
  one-dimensional population balance approach}.
\newblock PhD thesis, University of Strathclyde, 2018.

\bibitem{Ziff1980}
R.~M. Ziff.
\newblock Kinetics of polymerization.
\newblock {\em Journal of Statistical Physics}, 23(2):241--263, 1980.

\end{thebibliography}
	%%%%%%%%%%%%%%%%
	%%%%%%%%%%%%%%%%

\end{document}